\PassOptionsToPackage{linktocpage}{hyperref}
\documentclass[hidelinks,onefignum,onetabnum]{siamart251216}
\usepackage{amsfonts,amssymb}
\usepackage{mathrsfs,mathtools}
\usepackage{graphicx,subfigure,epstopdf}
\graphicspath{{./figures_cv/}}
\usepackage{latexsym,float,tikz}

\usepackage{algorithm}
\usepackage{algpseudocode}
\usepackage{float}

\newcommand{\KwData}[1]{\State \textbf{Input:} #1}
\newcommand{\KwResult}[1]{\State \textbf{Output:} #1}

\usepackage{adjustbox}
\usepackage{esint}
\usepackage{multirow}
\usepackage{ctable}

\usepackage{verbatim}
\usepackage[margin=2.5cm]{geometry}
\usepackage{enumerate}
\usepackage{pdflscape}
\usepackage[compact]{titlesec}
\usepackage{enumitem}
\usepackage{url}
\usepackage{bbm}

\numberwithin{equation}{section}
\newtheorem{remark}[theorem]{Remark}
\newtheorem{assumption}{Assumption}[section]

\newcommand{\Cov}{C_{\mathrm{ov}}}

\newcommand{\uu}{\boldsymbol{u}}
\newcommand{\vv}{\boldsymbol{v}}
\newcommand{\ww}{\boldsymbol{w}}
\newcommand{\xx}{\boldsymbol{x}}

\newcommand{\nn}{\boldsymbol{n}}
\newcommand{\ttt}{\boldsymbol{t}}
\newcommand{\zzz}{\boldsymbol{0}}
\newcommand{\zz}{\boldsymbol{z}}
\newcommand{\ff}{\boldsymbol{f}}

\newcommand{\Cpoin}[1]{{\rm C}_{\mathrm{M}}(#1)}

\newcommand{\dx}{\,\mathrm{d}x}
\newcommand{\Ce}{{\rm {C}}_{\mathrm{est}}}

\newcommand{\Const}[1]{{\rm C}_{\mathrm{#1}}}
\newcommand{\op}{\operatorname}
\newcommand{\Yueqi}[1]{{\color{black}#1}}

\begin{document}
\title{Numerical homogenization for indefinite time-harmonic Maxwell equations}
\author{Yueqi Wang\thanks{Department of Mathematics, Purdue University, 610 Purdue Mall, West Lafayette, 47907, IN, USA ({\tt{wang7406@purdue.edu}}). Part of YW's work was supported by Hong Kong RGC via the Hong Kong PhD Fellowship Scheme.} \and Wing Tat Leung \thanks{Department of Mathematics, City University of Hong Kong, 83 Tat Chee Ave, Kowloon Tong, Hong Kong SAR, P.R. China. ({\tt{wtleung27@cityu.edu.hk}}).} \and Guanglian Li\thanks{Corresponding author. Department of Mathematics, The University of Hong Kong, Pok Fu Lam Road, Hong Kong SAR, P.R. China. ({\tt{lotusli@maths.hku.hk}}) GL acknowledges the support from Hong Kong RGC General Research Fund (Project number: 17308924).}}
\maketitle
\begin{abstract}
We propose a novel numerical homogenization method based on the edge multiscale approach for solving indefinite time-harmonic Maxwell equations in heterogeneous media with large wavenumber. Numerical methods for these equations in homogeneous media with high wavenumber are particularly challenging due to the so-called pollution effect: the mesh size must be significantly smaller than the reciprocal of the wavenumber to achieve a desired accuracy. This challenge is amplified in heterogeneous media, which frequently occur in practical applications such as metamaterial simulations, since resolving the heterogeneity is necessary for obtaining reliable solutions. Our approach overcomes this difficulty by avoiding explicit resolution of the heterogeneity, while employing a mesh size that depends almost linearly on the reciprocal of the wavenumber. The approximation properties and stability of the method rely critically on the development and rigorous analysis of a novel, nonstandard variational formulation, which constitutes the main innovation of this work. Extensive numerical experiments are provided to validate our theoretical findings.
\end{abstract}
\begin{keywords} time-harmonic Maxwell equation, heterogeneous media, large wavenumber, numerical homogenization, multiscale methods
\end{keywords}
\begin{MSCcodes}
 35Q61, 65N12, 65N15, 65N30, 78M10
\end{MSCcodes}
\section{Introduction}
Let $D\subset\mathbb{R}^3$ be  open, bounded, contractible domain with polyhedral Lipschitz boundary. We consider Maxwell equations in the frequency domain with impedance boundary condition, and seek $\uu:\;D\to\mathbb{C}^3$ in a proper functional space, s.t., 
\begin{equation}\label{eq:model}
	\left\{
	\begin{aligned}
		\op{curl}\left(\epsilon^{-1}\op{curl}\uu\right)
		-k^2\uu&=\zzz&&\text{ in }D,\\
		\epsilon^{-1}\op{curl}\uu\times \nn-ik\uu_T&=\boldsymbol{g}&&\text{ on }\partial D.
	\end{aligned}
	\right.
\end{equation}
Here, $k$ is the wavenumber with $k\geq k_0>0$ with $k_0$ the lowest frequency considered in the analysis, $\nn$ is the unit outward normal vector on $\partial D$, $\uu_T\coloneqq(\nn\times\uu)\times \nn$ denotes tangential components trace and $\boldsymbol{g}$ the boundary data. We assume $\epsilon\in W^{1,\infty}(D)$ is uniformly positive with $1\leq\epsilon(\mathbf{x})
\leq\epsilon_{\max}$ for some upper bound $\epsilon_{\max}\in [1,\infty)$. Problems of this type arise from modeling the propagation of electromagnetic waves in heterogeneous media, and achieving efficient and reliable numerical simulations is of fundamental importance. Our focus is on cases where $\epsilon^{-1}$ varies rapidly and the wavenumber $k$ is large. It is important to note that the presence of multiple scales in the coefficient $\epsilon$ makes directly solving Problem \eqref{eq:model} particularly challenging, as resolving the problem at the finest scale would entail significant computational costs. Moreover, the high computational complexity involved in discretizing \eqref{eq:model}, especially for heterogeneous materials and large wavenumbers, is almost unavoidable. Therefore, multiscale model order reduction techniques are essential to mitigate the substantial computational and storage demands.

To address these challenges, many promising numerical methods have been developed and analyzed in the past few years. These include the edge finite element methods \cite{Cia-2016},  heterogeneous multiscale methods (HMM) with (locally) periodic structure \cite{ciarlet2017approximation,henning2016new,verfurth2019heterogeneous}, Localized Orthogonal Decomposition methods (LOD) \cite{gallistl2018numerical,doi:10.1137/19M1293818}, Constraint energy minimizaing Generalized Multiscale Finite Element Method (CEM-GMsFEM) \cite{chung2025multiscale}, GMsFEM \cite{MR3851624}, multiscale spectral generalized finite element method (MS-GFEM) \cite{ma2025curl} and immersed finite element method \cite{MR4678368}. Among these well-established approaches, most of them focus on $\mathbf{H}(\operatorname{curl})$-elliptic problems rather than the indefinite problems considered here. The indefinite case presents additional theoretical and numerical challenges, primarily due to difficulties in establishing the stability of the numerical schemes and the inherently indefinite nature of the discretized problem. Despite ongoing developments, progress in this area remains in its early stages, and much work is needed to effectively address practical, real-world applications. To the best of our knowledge, efficient multiscale methods that achieve sufficient accuracy for indefinite time-harmonic Maxwell equations with heterogeneous media are not yet available. This gap constitutes the main motivation and goal of this paper.

In this paper, we develop a novel edge multiscale method for \eqref{eq:model} such that the computational complexity grows almost linearly with respect to the wavenumber, and in the mean while, independent of the heterogeneity. 
The main idea of this method is to first represent the global solution as a superposition of local solutions supported on a family of overlapping subdomains associated with a coarse mesh that does not resolve the fine-scale heterogeneity. Each local solution is determined by its trace on the boundary of its subdomain. Then we approximate these traces using suitable basis functions defined on the coarse skeleton, which induces the multiscale ansatz space and establishes its approximation properties. This method was first proposed in \cite{MR3939320} and implemented for the elliptic problems with heterogeneous coefficients \cite{MR3980476}. The main contributions of this paper are as following. 1). We develop a novel edge multiscale method for time-harmonic Maxwell equation based upon local-global splitting of the solution (Algorithm \ref{algorithm:wavelet}). 2). The approximation properties of the associated multiscale ansatz space are rigorously established. This is achieved by first proposing a nonstandard variational formulation (Lemma \ref{lem:nonstandard-var}) and then analyzing the regularity of its solution in a proper functional space (Lemma \ref{lem:well-posedness-nonstandard} and Theorem \ref{lem:very-weak}). 3) The well-posedness of Algorithm \ref{algorithm:wavelet} is rigorously proven (Theorem \ref{wellpose}). 4) Extensive numerical experiments are conducted to validate our theoretical results.

The remainder of this paper is structured as follows. We first introduce in Section \ref{sec:model} the standard functional spaces for \eqref{eq:model}, and show its well-posedness and finite element approximation. Then we present our multiscale method in Section \ref{sec:main-method} based upon the local-global splitting of the solution.  The convergence of this multiscale method is presented in Section \ref{sec:error}. This is followed by extensive numerical test with homogeneous and heterogeneous coefficients in Section \ref{sec:num}. Finally, we summarize our main results and propose several future work in Section \ref{sec:conclusion}. 

Throughout the paper we adopt the standard notation for the complex-valued
Lebesgue space $L^{\infty}(G;\mathbb{C})$ and Sobolev spaces $H^{s}(G;\mathbb{C})$ for $s\in\mathbb{R}$ on an open bounded domain
$G\subset\mathbb{R}^{3}$. We suppress $\mathbb{C}$ for simplicity. Related norms and semi-norms of
$H^{s}(G)$ and the norm of $L^{\infty}(G)$ are denoted by $\|\cdot\|_{s,G}$, $|\cdot|_{s,G}$ and
$\|\cdot\|_{L^\infty(G)}$, respectively.
We define $(w,v)_{G}\coloneqq\int_Gw\overline{v}{\rm d}x$ the $L^2(G)$ inner product and denote $\|\cdot\|_{L^2(G)}$ its induced norm, $\langle w,v\rangle_{\gamma}\coloneqq\int_{\gamma}w\overline{v}{\rm d}s$ the $L^2(\gamma)$ inner product and denote $\|\cdot\|_{L^2(\gamma)}$ its induced norm
for any $G\subset D$, and $\gamma\subset\partial G$. We use $\Re$, $\Im$ and $\overline{\cdot}$ to denote the real part, imaginary part and conjugate of a complex number. The notation $a\lesssim b$ means that there exists constant $C>0$, independent of all parameters of interest ($k,H,\epsilon$), such that $a\leq Cb$. We also write $a\simeq b$ if $a \lesssim b$ and $b \lesssim a$.

\section{Problem setting}\label{sec:model}
We first introduce several standard functional spaces for \eqref{eq:model}, and then summarize key analytical properties, discuss the associated numerical challenges, and present its finite element approximation.

We start with several Hilbert spaces that commonly used in the Maxwell equations. 
$\mathbf{H}(\operatorname{curl};D)\coloneqq\{\ww\in (L^2(D))^3|\op{curl}\ww\in (L^2(D))^3\}$, $\mathbf{H}_0(\operatorname{curl};D)\coloneqq\mathrm{closure\;of}\;
(\mathcal{C}_0^{\infty}(D))^3\;\mathrm{in}\;\mathbf{H}(\operatorname{curl};D)$, 
$\mathbf{H}_{\operatorname{imp}}(\operatorname{curl};D)\coloneqq\{\ww\in \mathbf{H}(\operatorname{curl};D)|\ww_T\in (L^2(\partial D))^3\}$, 
$L^2_{\ttt}(\partial D)\coloneqq\{\boldsymbol{v}\in (L^2(\partial D))^3|\boldsymbol{v}\cdot\nn=0\; \mathrm{ on }\;\partial D\}$, 
$\mathbf{H}_{\ttt}^{s}(\partial D)\coloneqq({H}^{s}(\partial D))^3\cap L^2_{\ttt}(\partial D)$ for $s\in \left[0,\frac{1}{2}\right)$, 
$\mathbf{H}(\op{div }0;D)\coloneqq\{\vv\in (L^2(D))^3|\nabla\cdot\vv=0\text{ in }D\}$, 
$\mathbf{H}_0(\op{div };D)\coloneqq\{\vv\in \mathbf{H}(\op{div };D)|\vv\cdot\nn=0\text{ on }\partial D\}$, 
$\mathbf{H}_0(\op{div }0;D)\coloneqq\{\vv\in \mathbf{H}(\op{div }0;D)|\vv\cdot\nn=0\text{ on }\partial D\}$, 
$\mathbf{H}^s(\op{curl};D)\coloneqq\{\boldsymbol{v}\in ({H}^s(D))^3|\op{curl}\boldsymbol{v}\in ({H}^s(D))^3\}$ for $s\geq 0$. $\ww_T\coloneqq(\nn\times\ww)\times \nn$ denotes tangential components trace of $\ww\in \mathbf{H}(\operatorname{curl};D)$. We endow $\mathbf{H}_{\op{imp}}(\mathrm{curl};D)$ with the $k$-weighted graph norm $\Vert\cdot\Vert_{\mathbf{H}_{\op{imp}}(\mathrm{curl};D)}$ given by
\begin{equation*}
\Vert\boldsymbol{u}\Vert_{\mathbf{H}_{\op{imp}}(\mathrm{curl};D)}^2\coloneqq\Vert\op{curl}\boldsymbol{u}\Vert^2_{L^2(D;\epsilon^{-1/2})}+k^2\Vert\boldsymbol{u}\Vert^2_{L^2(D)}+k\Vert\boldsymbol{u}_T\Vert^2_{L^2(\partial D)},
\end{equation*} 
where $\|\ff\|_{L^2(\cdot;\epsilon^{-1/2})}:=\|\epsilon^{-1/2}\ff\|_{L^2(\cdot)}$ denotes the weighted $L^2$ norm.


Then we are ready to provide a precise formulation for \eqref{eq:model}. Let $\boldsymbol{g}\in L^2_{\ttt}(\partial D)$, we aim to 
find $\uu\in \mathbf{H}_{\op{imp}}(\operatorname{curl};D)$ such that
\begin{equation*}
	\left\{
	\begin{aligned}
		\op{curl}\left(\epsilon^{-1}\op{curl}\uu\right)
		-k^2\uu&=\zzz&&\text{ in }D,\\
		\mathcal{B}\uu\coloneqq\epsilon^{-1}\op{curl}\uu\times \nn-ik\uu_T&=\boldsymbol{g}&&\text{ on }\partial D.
	\end{aligned}
	\right.
\end{equation*}

\subsection{Weak formulation and {\it a priori } estimate}\label{priori estimate}
We define the sesquilinear form $a(\cdot,\cdot)$ on $\mathbf{H}_{\op{imp}}(\op{curl};D)\times \mathbf{H}_{\op{imp}}(\op{curl};D)$ associated to~\eqref{eq:model} by
\begin{align}\label{eq:sesqu}
	a(\boldsymbol{u},\boldsymbol{v}):=(\epsilon^{-1}\op{curl} \boldsymbol{u}, \op{curl}\boldsymbol{v})_{D}-k^2( \boldsymbol{u},\boldsymbol{v})_{D}
	-ik\langle\uu_T,\vv_T\rangle_{\partial D}.
\end{align}
The sesquilinear form $a(\cdot,\cdot):\mathbf{H}_{\op{imp}}(\op{curl};D)\times\mathbf{H}_{\op{imp}}(\op{curl};D)\to\mathbb{C}$ is bounded, 
\begin{align}\label{a property 1}
|a(\ww,\vv)|\leq\Vert\ww\Vert_{\mathbf{H}_{\op{imp}}(\op{curl};D)}\Vert\vv\Vert_{\mathbf{H}_{\op{imp}}(\op{curl};D)}\;\forall
\ww,\vv\in \mathbf{H}_{\op{imp}}(\op{curl};D).		
\end{align}
Moreover, there holds the G$\mathring{a}$rding's inequality, 
\begin{align}\label{a property 2}
|a(\ww,\ww)|+2k^2(\ww,\ww)_{L^2(D)}\geq\Vert\ww\Vert_{\mathbf{H}_{\op{imp}}(\op{curl};D)}^2 \;\forall\ww\in \mathbf{H}_{\op{imp}}(\op{curl};D).
\end{align}
Then the weak formulation to \eqref{eq:model} reads as seeking $\boldsymbol{u}\in \mathbf{H}_{\op{imp}}(\mathrm{curl};D)$ s.t.,
\begin{equation}\label{eq:bilinear-form}
a(\boldsymbol{u},\boldsymbol{v})=\langle\boldsymbol{g},\vv_T\rangle_{\partial D}
\;\forall\boldsymbol{v}\in \mathbf{H}_{\op{imp}}(\mathrm{curl};D).
\end{equation}
The existence and uniqueness of solution to \eqref{eq:bilinear-form} is proved with Fredholm theory; see \cite[Theorem 13]{haddar2015maxwell} and \cite[Theorem 4.17]{Monk2003}. Its stability with respect to the wavenumber relies on the geometry of the domain. 
For $C^2$-domains with smooth boundary data $\boldsymbol{g}\in \mathbf{H}_{\ttt}^{1/2}(\partial D)$ and Lipschitz-continuous coefficient, it is well-known that $\uu\in \mathbf{H}^1(\op{curl};D)$ \cite[Section 4.5.d]{costabel2010corner}. For the general case, i.e., $\boldsymbol{g}\in L^2_{\ttt}(\partial D)$, the regularity of the solution $\boldsymbol{u}$ is reduced. We introduce one critical assumption on the regularity of the solution $\uu$, which is valid for some specific data $\boldsymbol{g}$, see, for example, \cite[Theorems 4.1 and 4.7]{chicaud2023analysis}.
\begin{assumption}\label{ass:regularity}
	Let $\uu\in \mathbf{H}_{\op{imp}}(\mathrm{curl};D)$ be the solution to Problem \eqref{eq:bilinear-form}, we assume $\uu\in \mathbf{H}^{s}(\op{curl};D)$ for some $s\in(\frac{1}{2},1]$. 
\end{assumption}

Next, we introduce a dual problem to \eqref{eq:bilinear-form}. For any ${\vv}\in (L^2(D))^3$, let $\zz\in \mathbf{H}_{\op{imp}}(\op{curl};D)$ satisfy
\begin{align}\label{dual}
	a(\boldsymbol{w},\zz)=(\ww,\vv)_{D}\;\forall\boldsymbol{w}\in\mathbf{H}_{\op{imp}}(\mathrm{curl};D). 
\end{align}
Moreover, we assume for some $s\in (\frac{1}{2},1]$ and $\theta\geq 0$, there holds
\begin{align}\label{ass:dual}
	\Vert\zz\Vert_{\mathbf{H}^{s}(\op{curl};D)}\leq \Const{stab}k^{\theta}\Vert\vv\Vert_{L^2(D)}.
\end{align}
Here, $\|\vv\|_{\mathbf{H}^{s}(\op{curl};D)}\coloneqq \|\op{curl}\vv\|_{H^{s}(D)}+k\|\vv\|_{H^{s}(D)}$ for all $\vv\in \mathbf{H}^{s}(\op{curl};D)$ and $s\geq 0$.
\subsection{Finite element discretization}\label{subsec:fem}
To discretize problem \eqref{eq:bilinear-form},
let $\mathcal{T}_H$ be a shape-regular, conformal hexahedral mesh of the physical domain $D$ with a mesh size $H$. We assume that $\epsilon$ is rapidly varying in $D$, and that the coarse mesh $\mathcal{T}_H$ fails to resolve this phenomena. Furthermore, we assume that all elements are parallelepipeds with edges parallel to the coordinate axes. This implies that each element $K\in\mathcal{T}_H$ can be obtained from the reference element $\hat{K}=(0,1)^3$ via a diagonal affine map $F_K(\mathbf{x})=\mathbf{B}_K\mathbf{x}+\mathbf{b}$ with $\mathbf{B}_K$ an invertible diagonal matrix. 
The vertices of $\mathcal{T}_H$
are denoted by $\{O_i\}_{i=1}^{N}$, with $N$ being the total number of coarse nodes.
The coarse neighborhood associated with the node $O_i$ is denoted by 
$\omega_i:=\bigcup\overline{\{ K_j\in\mathcal{T}_H: ~~~ O_i\in \overline{K}_j\}}$.
The overlapping constant $\Cov$ is defined by
\begin{align}\label{eq:overlap}
	\Cov:=\max\limits_{K\in \mathcal{T}_{H}}\#\{O_i: K\subset\omega_i \text{ for } i=1,2,\cdots,N\}.
\end{align}
Moreover, we denote $\delta_i\coloneqq\max\{\|\nabla\epsilon^{-1}\|_{L^{\infty}(\omega_i)},1\}$, $\bar{\epsilon}_{i}:=
\|\epsilon\|_{L^{\infty}(\omega_i)}$ and $\underline{\epsilon}_{i}:=\|\epsilon^{-1}\|_{L^{\infty}(\omega_i)}^{-1}$. Additionally, we introduce the local contrast $\Lambda_i:=\bar{\epsilon}_{i}\underline{\epsilon}_{i}^{-1}$.

Over the coarse mesh $\mathcal{T}_H$, let $\mathbf{V}_H$ be the $\mathbf{H}(\mathrm{curl})$-conforming finite element space,
\[
\mathbf{V}_H:=\{\mathbf{v}\in \mathbf{H}(\mathrm{curl};D): \vv|_{K}\in  Q_{p-1,p,p}\times Q_{p,p-1,p}\times Q_{p,p,p-1} \text{ for all } K\in \mathcal{T}_H\},
\]
for given integer $p\geq 1$. Here, ${Q}_{j_1,j_2,j_3}$ denotes the space of polynomial functions with polynomial of degree at most $j_k$ with respect to $x_k$ for $k=1,2,3$. 
Let $\mathbf{I}_H: H(\op{curl};D)\to V_H$ be the interpolant defined in \cite[Section 6]{Monk2003} or \cite[Section 2]{bulovyatov2010parallel}. The interpolation errors are well-known \cite[Theorem 6.6]{Monk2003}.
Let $\vv\in \mathbf{H}^s(\op{curl};D)$ for $s\in (\frac{1}{2},1]$, then there holds	
\begin{align*}
		\Vert \boldsymbol{v}-\mathbf{I}_H\boldsymbol{v}\Vert_{L^2(D)}+\Vert \op{curl}(\boldsymbol{v}-\mathbf{I}_H\boldsymbol{v})\Vert_{L^2(D)}\lesssim H^s\left(\Vert \boldsymbol{v}\Vert_{\mathbf{H}^s(D)}+\Vert \op{curl}\boldsymbol{v}\Vert_{\mathbf{H}^s(D)}\right).
	\end{align*}
Then the standard $\mathbf{H}(\op{curl})$-conforming Galerkin FEM approximation of Problem \eqref{eq:bilinear-form} is to find $\boldsymbol{u}_H\in \mathbf{V}_H$, s.t., 
\begin{align}\label{eqn:weakform_h}
	a(\boldsymbol{u}_H,\boldsymbol{v}_H)=\langle\boldsymbol{g},\boldsymbol{v}_H\rangle_{\partial D} \quad \forall \boldsymbol{v}_H\in \mathbf{V}_H.
\end{align}
The existence of the finite element solution of \eqref{eqn:weakform_h} has been proved provided that the mesh size $H$ is sufficiently small \cite[Lemma 9.6]{ melenk2023wavenumber}. 
Let $D$ be a bounded Lipschitz domain with simply connected and sufficiently smooth boundary, there exists a constant $C_1, C_2$ independent of $k$ and $H$, such that if the mesh satisfy
\begin{equation}\label{inf sup1}
\frac{Hk}{p}\leq C_1,\text{ and }p\geq\max\{1,C_2\ln k\},
\end{equation}
then the discrete problem \eqref{eqn:weakform_h} has a unique solution, which satisfies the quasi-optimal error estimate
\begin{equation*}
\Vert\boldsymbol{u}-\boldsymbol{u}_{H}\Vert_{\mathbf{H}_{\op{imp}}(\mathrm{curl};D)}
	\lesssim\inf\limits_{\boldsymbol{v}\in  \mathbf{V}_{H}}\Vert \boldsymbol{v}-\boldsymbol{u}\Vert_{\mathbf{H}_{\op{imp}}(\mathrm{curl};D)}.	
\end{equation*}
\section{The algorithm}\label{sec:main-method}
We introduce in this section the methodology for the construction of multiscale ansatz space $\mathbf{V}_{\text{ms},\ell}$.
\subsection*{Local-global splitting}
To start, we define the local sesquilinear form $a_i(\cdot,\cdot)$ on $\mathbf{H}_{\op{imp}}(\op{curl};\omega_i)\times \mathbf{H}_{\op{imp}}(\op{curl};\omega_i)$ by
\begin{align*}
	a_i(\boldsymbol{u},\boldsymbol{v})\coloneqq (\epsilon^{-1}\op{curl} \boldsymbol{u}, \op{curl}\boldsymbol{v})_{\omega_i}-k^2( \boldsymbol{u},\boldsymbol{v})_{\omega_i}
	-ik\langle\uu_T,\vv_T\rangle_{\partial \omega_i}.
\end{align*}
Note that the restriction $\uu^{i}\coloneqq\uu|_{\omega_i}\in \mathbf{H}_{\op{imp}}(\op{curl};\omega_i)$ satisfies
\begin{equation}\label{eq:loc-splitting}
a_i(\uu^{i},\vv)=\langle(\mathcal{B}\uu)|_{\gamma_i},\vv_T\rangle_{\gamma_i}+\langle\boldsymbol{g},\vv_T\rangle_{\Gamma_i} \qquad \forall \vv\in H_{\op{imp}}(\op{curl};\omega_i).
\end{equation}
Here, $\gamma_i\coloneqq \partial\omega_i\backslash\partial D$ and $\Gamma_i\coloneqq \partial\omega_i\cap \partial D$ denote the internal and outer boundaries for each subdomain $\omega_i$. 

The linearity of \eqref{eq:loc-splitting} implies the decomposition $\uu^{i}\coloneqq \uu^i_B+\uu^i_I$, with $\uu^i_B\in \mathbf{H}_{\op{imp}}(\op{curl};\omega_i)$ as the local boundary particular solution determined by the global boundary condition $\boldsymbol{g}$,
\begin{equation}\label{eq:loc-splitting-b}
a_i(\uu^{i}_B,\vv)=\langle\boldsymbol{g},\vv_T\rangle_{\Gamma_i} \qquad \forall \vv\in \mathbf{H}_{\op{imp}}(\op{curl};\omega_i),
\end{equation}
and local multiscale basis function $\uu^i_I\in \mathbf{H}_{\op{imp}}(\op{curl};\omega_i)$ that is independent of the global boundary condition $\boldsymbol{g}$, 
\begin{equation}\label{eq:loc-splitting_I}
a_i(\uu^{i}_I,\vv)=\langle(\mathcal{B}\uu)|_{\gamma_i},\vv_T\rangle_{\gamma_i} \qquad \forall \vv\in \mathbf{H}_{\op{imp}}(\op{curl};\omega_i).
\end{equation}

Next, we introduce a partition of unity $\{\chi_i\}_{i=1}^N$ subordinate to the cover $\{\omega_i\}_{i=1}^N$ to induce a global decomposition, which is assumed to satisfy 
\begin{equation}\label{eq:pou}
\begin{aligned}
{\text{supp}(\chi_i)}\subset\overline{\omega}_i,\;
\sum_{i=1}^{N}\chi_i=1 \text{ in } D,\;
 \| \chi_i\|_{C(\overline{\omega}_i)}\leq 1,\;
\|\nabla \chi_i\|_{C(\overline{\omega}_i)}\leq C_{\text{G}}H^{-1}
\end{aligned}
\end{equation}
for some positive constant $ C_{\text{G}}$. Moreover, we assume this partition of unity $\{\chi_i\}_{i=1}^N$ satisfies the vanishing gradient property, 
\begin{align}\label{prop:vanish-grad}
\nabla\chi_i=\zzz \text{ on }\partial \omega_i.
\end{align}
This property is crucial for the definition of the boundary multiscale basis functions, which can be guaranteed by known result, e.g., the flat-top partition of unity functions defined by the Shepard's functions \cite{Schweitzer:1997}. 

Consequently, we can represent $\uu$ as a summation of local parts,
\begin{equation}\label{eq:harmonic-glo}
\begin{aligned}
\uu&=\left(\sum_{i=1}^{N}\chi_i\right) \uu=\sum_{i=1}^{N}\left(\chi_i \uu^{i}\right)\\
&=\sum_{i=1}^{N}\chi_i (\uu^{i}_B+\uu^i_I)=\sum_{i=1}^{N}\left(\chi_i \uu^{i}_B\right)+
\sum_{i=1}^{N}\left(\chi_i \uu^i_I\right)=:\uu_B+\uu_I.
\end{aligned}
\end{equation}
Here, $\uu_B$ denotes a global particular solution that is determined by the global boundary condition $\boldsymbol{g}$ and thus known, and $\uu_I$ denotes the unknown function we want to approximate. Note that \eqref{eq:loc-splitting-b} and \eqref{prop:vanish-grad} implies $\uu_B$ satisfy
\begin{align*}
\mathcal{B} \uu_B=\boldsymbol{g}\quad \mbox{on }\partial D.
\end{align*}
\subsection*{Haar wavelet}

Let the scaling function $\phi(x)$ and the mother wavelet $\psi(x)$ be given by
\begin{equation*}
	\phi(x)=
	\left\{
	\begin{aligned}
		&1, &&\text{ if } 0\leq x\leq 1,\\
		&0, &&\text{ otherwise,}
	\end{aligned}
	\right.
	\qquad
	\psi(x)=
	\left\{
	\begin{aligned}
		&1, &&\text{ if } 0\leq x\leq 1/2,\\
		&-1, &&\text{ if }1/2< x\leq 1,\\
		&0, &&\text{ otherwise}.
	\end{aligned}
	\right.
\end{equation*}
By means of dilation and translation, the mother wavelet $\psi(x)$ can
result in an orthogonal decomposition of the space $L^2(I)$ with $I:=[0,1]$. To this end, we can define the basis functions on level $\ell\geq 1$ by
\begin{align*}
	\psi_{\ell,j}(x):=2^{\frac{\ell-1}{2}}\psi(2^{\ell-1}x-j) \quad \text{ for all }\quad 0\leq j\leq 2^{\ell-1}-1.
\end{align*}
The subspace $W_{\ell}$ of level $\ell$ is
\begin{equation*}
	W_{\ell}:=
	\left\{
	\begin{aligned}
		&\text{span}\{\phi\} &&\text{ for }\ell =0\\
		&\text{span}\{\psi_{\ell,j}:\quad 0\leq j\leq 2^{\ell-1}-1\}&&\text{ for }\ell\geq 1.
	\end{aligned}
	\right.
\end{equation*}
and we note that subspace $W_{\ell}$ is orthogonal to $W_{\ell'}$ in
$L^2(I)$ for any two different levels $\ell\neq \ell'$. We denote the
subspace in $L^2(I)$, up to level $L$, by $V_{L}$  defined by
\[
V_{L}:=\oplus_{m\leq L}W_{m}.
\]
Note that one can derive the hierarchical decomposition of the space
$L^2(I^{d-1})$ for $d>1$ by means of the tensor product, which is denoted as $V_{\ell}^{\otimes^{d-1}}$. 
Let ${P}_{\ell}: I^2(I^{d-1})\to V_{\ell}^{\otimes^{d-1}}$ be $L^2$-projection for each level $\ell\geq 0$ and let $s\in (0,1]$. Then there holds \cite[Proposition 3.1]{MR3980476}
\begin{align}\label{prop:approx-wavelets}
\|v-P_{\ell}v\|_{L^2(I^{d-1})}&\lesssim 2^{-s\ell}|v|_{H^s(I^{d-1})} &\;\forall v\in H^s(I^{d-1}).
\end{align}

\subsection*{Multiscale ansatz space}

Now we are ready to define the approximative space $\mathbf{V}_{\text{ms},\ell}$. 

To begin, we first define the approximation space over $\partial\omega_i$ for each subdomain $\omega_i$. Let the level parameter $\ell\in \mathbb{N}$ be fixed, and let $\partial\omega_{i}^k$ with $k=1,\cdots,m_i$ be the $m_i$ faces of $\omega_i$ as a partition of $\partial\omega_i$ with no mutual intersection, i.e., $\cup_{k=1}^{m_i}\overline{\partial\omega_{i}^k}=\partial\omega_i$ and $\partial\omega_{i}^k\cap \partial\omega_{i}^{k'}=\emptyset$ if $k\neq k'$. 

Consider the parallelogram-shaped face $\partial\omega_{i}^k$.  Let its two adjacent edges be denoted by the vectors $\mathbf{e}_1^{i,k}$ and $\mathbf{e}_2^{i,k}$. The corresponding unit tangential vectors along these edges are defined as
\begin{align*}
\ttt^{i,k}_1\coloneqq \frac{\mathbf{e}_1^{i,k}}{\|\mathbf{e}_1^{i,k}\|_2} \text{ and } \ttt^{i,k}_2\coloneqq \frac{\mathbf{e}_2^{i,k}}{\|\mathbf{e}_2^{i,k}\|_2}
\end{align*}
with $\|\cdot\|_2$ the Euclidean norm in $\mathbb{R}^3$. 
These unit vectors define the local orientation and metric of the face. Their cross product, $\ttt^{i,k}_1\times \ttt^{i,k}_2$, is orthogonal to the face, and its magnitude is related to the sine of the interior angle between the edges. We denote $\mathbf{V}_{i,\ell}^k\subset C(\partial\omega_{i}^{k})$ as the space formulated by the Haar wavelets along the tangential directions on each face $\partial\omega_{i}^{k}$ up to level $\ell$, i.e., 
\begin{equation*}
	\mathbf{V}_{i,\ell}^k\coloneqq\text{span}\{\Psi^{i,k}_{l,q}\ttt_n^{i,k}:1\leq l,q\leq2^{\ell}\text{ and }  1\leq n\leq 2\},
\end{equation*}
where $\Psi^{i,k}_{l,q}\coloneqq\psi_{l,1}^{i,k}\psi_{q,2}^{i,k}$, and $\psi_{l,\ast}^{i,k}$ is the Haar wavelet along the edge $\mathbf{e}_{\ast}^{i,k}$ with level up to $\ell$. Then the local face space $\mathbf{V}_{i,\ell}$ defined over $\partial\omega_i$ is the smallest space having $\mathbf{V}_{i,\ell}^k$ as a subspace, which can be represented by  
\begin{align}\label{eq:local-edge}
	\mathbf{V}_{i,\ell} := \oplus_{k=1}^{m_i}\mathbf{V}_{i,\ell}^k.
\end{align}
Next, we introduce the local multiscale space over each coarse neighborhood $\omega_i$, 
\begin{align}\label{eq:local-multiscale}
	\mathcal{L}^{-1}_{i}(\mathbf{V}_{i,\ell})
	:= \text{span} \left\{\mathcal{L}^{-1}_{i,n}(\Psi^{i,k}_{l,q}): \,  \,  1\leq l,q\leq2^{\ell}, 1\leq n\leq2, 1\leq k\leq m_i\right\}.
\end{align}
Here, $\mathcal{L}^{-1}_{i,n} (\Psi^{i,k}_{l,q}):=\ww_{l,q,n}^{i,k}\in \mathbf{H}_{\op{imp}}(\mathrm{curl};\omega_i)$, is defined by
\begin{equation}
	\label{eq:Li 2}
	a_i(\ww_{l,q,n}^{i,k},\vv)=\langle\Psi^{i,k}_{l,q}\ttt_n^{i,k},\vv_T\rangle_{\partial\omega_i^k}\qquad\forall\vv\in H_{\op{imp}}(\op{curl};\omega_i).
	\end{equation}
Finally, the multiscale ansatz space is defined by the Partition of Unity $\{\chi_i\}_{i=1}^N$, which is $\mathbf{H}(\op{curl};D)$-conforming,
\begin{align}\label{eq:global-multiscale}
	\mathbf{V}_{\text{ms},\ell} := \text{span} \left\{\chi_i\mathcal{L}^{-1}_{i}(\mathbf{V}_{i,\ell}) : \,  \,  1 \leq i \leq N\right\}.
\end{align}
\subsection*{$\mathbf{H}(\op{curl};D)$-conforming Galerkin formulation}
We seek $\uu^I_{\text{ms},\ell}\in \mathbf{V}_{\text{ms},\ell}$, s.t.,
\begin{align}\label{eqn:weakform_h-romb}
	a\left(\uu^I_{\text{ms},\ell},\vv_{\text{ms},\ell}\right)=\langle\boldsymbol{g},\vv_{\text{ms},\ell}\rangle_{\partial D}-a\left(\uu_B,\vv_{\text{ms},\ell}\right)
	\quad \forall \vv_{\text{ms},\ell}\in \mathbf{V}_{\text{ms},\ell}.
\end{align}
Then the approximation to $\uu$ is
\begin{align}\label{eq:ms-soln}
\uu_{\text{ms},\ell}\coloneqq\uu_B+\uu^I_{\text{ms},\ell}.
\end{align}
Our main algorithm is summarized in Algorithm \ref{algorithm:wavelet}.

\begin{algorithm}[H]
	\caption{Wavelet-based Edge Multiscale Finite Element Method (WEMsFEM)}
	\label{algorithm:wavelet}
	\begin{algorithmic}[1]
		\KwData{The level parameter $\ell\in \mathbb{N}$; coarse neighborhood $\omega_i$ and its $m_i$ mutually disjoint faces $\partial\omega_{i}^{k}$; the subspace $\mathbf{V}_{i,\ell}^k\subset L^2(\partial\omega_{i}^{k})$ up to level $\ell$ on each coarse face $\partial\omega_{i}^{k}$.}
		\KwResult{$\uu_{\text{ms},\ell}$}
		
		\State Construct the local face space $\mathbf{V}_{i,\ell}$ \eqref{eq:local-edge}.
		\State Calculate the local multiscale space $\mathcal{L}^{-1}_{i}(\mathbf{V}_{i,\ell})$ \eqref{eq:local-multiscale}.
		\State Construct the global multiscale space $\mathbf{V}_{\text{ms},\ell}$ \eqref{eq:global-multiscale}.
		\State Calculate the local particular solutions $\{\uu^{i}_B\}_{\Gamma_i\neq\emptyset}$ \eqref{eq:loc-splitting-b} and formulate the global particular solution
		$\uu_B=\sum_{\Gamma_i\neq\emptyset}\left(\chi_i \uu^{i}_B\right)$.
		\State Solve for $\uu^I_{\text{ms},\ell}\in \mathbf{V}_{\text{ms},\ell}$ from \eqref{eqn:weakform_h-romb}.
		\State Formulate the approximation $\uu_{\text{ms},\ell}=\uu_B+\uu^I_{\text{ms},\ell}$.
	\end{algorithmic}
\end{algorithm}

\section{Error estimation}\label{sec:error}

We present the error bound for Algorithm \ref{algorithm:wavelet}, which is divided into four steps.
First, we derive an {\it a priori} estimate for the local problem \eqref{eq:very-weak}, establishing the local approximation properties of the local multiscale space $\mathcal{L}^{-1}_i(\mathbf{V}_{i,\ell})$.
Second, we derive the approximation properties of the edge multiscale ansatz space $\mathbf{V}_{\text{ms},\ell}$ \eqref{eq:global-multiscale}. Third, we provide the global
approximation properties of the ansatz space in Corollary \ref{thm:last}. Finally, the error bound of Algorithm \ref{algorithm:wavelet} is presented in Theorem \ref{wellpose}.
\subsection{Local error estimate}\label{sec:appendix}
Let $\omega_i$ be a coarse neighborhood for any
$i=1,\cdots,N$ and let 
$X_T(\omega_i)\coloneqq \mathbf{H}_0(\operatorname{curl};\omega_i)\cap \mathbf{H}(\operatorname{div }0;\omega_i)$. Define
\begin{align*}
\Cpoin{\omega_i}\coloneqq H^{-2}\sup\limits_{\boldsymbol{z}\in X_T(\omega_i)\backslash\{\zzz\}}\frac{\Vert\boldsymbol{z}\Vert_{L^2(\omega_i)}^2}
{\Vert\op{curl}\boldsymbol{z}\Vert_{L^2(\omega_i)}^2}.
\end{align*}
Then the positive constants $\Cpoin{\omega_i}$ is independent of 
the coarse mesh $\mathcal{T}_{H}$. Note that we will utilize the same
constant $\Cpoin{\omega_i}$ to denote the constant from the so-called Maxwell estimate for vector fields, cf.  \cite[Lemma 1]{neff2012canonical} and \cite[Corollary 3.2]{schweizer2018friedrichs}.

\begin{assumption}[Scale Resolution Assumption]\label{ass:resolution}
The coarse mesh size $H$ is sufficiently small, s.t., 
\begin{align*}
\max_{i=1,2,\ldots,N}\Cpoin{\omega_i}\bar{\epsilon}_{i}(Hk)^{2}<1.
\end{align*}
\end{assumption}
Under Assumption \ref{ass:resolution}, we introduce a constant independent of $H$ and $k$, given by 
\begin{align}\label{ass:res-constant}
\Ce&:=\left(1-\max_{i=1,2,\ldots,N}\Cpoin{\omega_i}\bar{\epsilon}_{i}(Hk)^2\right)^{-1}.
\end{align}
Definition \eqref{ass:res-constant} together with Assumption \ref{ass:resolution} implies $\Ce>1$. 

Next, we introduce a nonstandard variational formulation 
in the spirit of the transposition method \cite{MR0350177} (for second-order elliptic
PDEs of divergence form) for the following problem,
\begin{equation}\label{eq:very-weak}
\left\{
\begin{aligned}
\mathcal{L}_i\vv&=\zzz&&\text{ in }\omega_i\\
\mathcal{B} \vv&=\ff &&\text{ on }\partial\omega_i,
\end{aligned}
\right.
\end{equation}
where the boundary data $\ff\in L^2_{\ttt}(\partial\omega_i)$, i.e., $\ff\in L^2(\partial\omega_i)^3$ and $\ff\cdot\nn=0$ on $\partial\omega_i$, and $\mathcal{L}_i\coloneqq\op{curl}\left({\epsilon^{-1
}}\op{curl}\cdot\right)-k^2$. We aim to establish its well-posedness in a suitable functional space and derive an {\it a priori} estimate. 

To begin with, we introduce several functional spaces. 
\begin{align*}
\mathbf{H}(\mathrm{curl};\partial\omega_i):=\left\{\boldsymbol{z}\in \mathbf{H}(\operatorname{div }0;\omega_i):\;\nn\times\op{curl}\boldsymbol{z}\in L^2(\partial\omega_i)^3 \right\}, 
\end{align*}
 which is a subspace of $L^2(\omega_i)^3$  associated with norm 
\begin{align*}
\Vert\uu\Vert_{\mathbf{H}(\mathrm{curl};\partial\omega_i)}^2\coloneqq k\Vert\uu\Vert_{L^2(\omega_i)}^2+\Vert\nn\times\op{curl}\uu\Vert_{L^2(\partial\omega_i;\epsilon^{-1})}^2,
\end{align*}
where $\|\ff\|_{L^2(\cdot;\epsilon^{-1})}\coloneqq\|\epsilon^{-1}\ff\|_{L^2(\cdot)}$ denotes the weighted $L^2$ norm.
Then we define a local test space $X(\omega_i)$ by
\begin{align}\label{eq:test-space}
X(\omega_i):=\left\{\zz\in \mathbf{H}_0(\operatorname{curl};\omega_i): \mathcal{L}_i\zz\in\mathbf{H}(\mathrm{curl};\partial\omega_i)\right\},
\end{align}
endowed with the norm $\|\zz\|_{X(\omega_i)}^2\coloneqq\Vert\op{curl} \zz\Vert^2_{L^2(\omega_i;\epsilon^{-1/2})}
+\Vert\mathcal{L}_i\zz\Vert_{\mathbf{H}(\mathrm{curl};\partial\omega_i)}^2$ for any $\zz\in X(\omega_i)$.

Next, we introduce a sesquilinear form $c(\cdot,\cdot):\mathbf{H}(\mathrm{curl};\partial\omega_i)\times \mathbf{H}(\mathrm{curl};\partial\omega_i)\to \mathbb{C}$, and a linear form $b(\cdot)$ on $\mathbf{H}(\mathrm{curl};\partial\omega_i)$, defined by
\begin{align*}
c(\boldsymbol{w}_1,\boldsymbol{w}_2)&\coloneqq\int_{\omega_i}\boldsymbol{w}_1\cdot \boldsymbol{\overline{w}}_2\;\dx-ik^{-1}\int_{\partial\omega_i}\epsilon^{-2}(\nn\times\op{curl}\boldsymbol{w}_1)\cdot (\nn\times\op{curl}\overline{\zz(\ww_2)})\;\mathrm{d}s\\
b(\boldsymbol{z}(\boldsymbol{w}))&\coloneqq ik^{-1}\int_{\partial\omega_i}\epsilon^{-1}\ff\cdot (\nn\times\op{curl}\overline{\zz(\ww)})\;\mathrm{d}s.
\end{align*}
Here, for all $\ww\in\mathbf{H}(\mathrm{curl};\partial\omega_i)$, $\boldsymbol{z}(\ww)\in X(\omega_i)$ denotes the unique solution, satisfying
\begin{equation}\label{eq:pde-dual}
\left\{\begin{aligned}
\mathcal{L}_i (\boldsymbol{z}(\ww))&=\boldsymbol{w} && \text{ in } \omega_i,\\
\boldsymbol{z}(\ww)\times \nn&=\zzz &&\text{ on }\partial \omega_i,
\end{aligned}\right.
\end{equation}
where $\nn$ denotes the unit outward normal on $\partial \omega_i$. 

Hence, we introduce a nonstandard variational formulation for \eqref{eq:very-weak}.
\begin{lemma}\label{lem:nonstandard-var}
Let $\vv\in \mathbf{H}(\mathrm{curl};\partial\omega_i)$ be the solution to Problem \eqref{eq:very-weak}. Then it satisfies
\begin{equation}\label{eq:veryweak-weakform}
	\begin{aligned}
		c(\vv,\ww)=b(\zz(\ww)) \qquad\forall \ww\in \mathbf{H}(\mathrm{curl};\partial\omega_i).
	\end{aligned}
\end{equation}
Here, $\boldsymbol{z}(\ww)\in X(\omega_i)$ is defined in \eqref{eq:pde-dual}.
\end{lemma}
\begin{proof}
Testing \eqref{eq:very-weak} with $\boldsymbol{z}(\ww)$ and applying integration by parts, we obtain
	\begin{equation*}
		\begin{aligned}
			\int_{\omega_i}\mathcal{L}_i\boldsymbol{v}\cdot \overline{\boldsymbol{z}(\ww)} \;\dx&= \int_{\omega_i}\epsilon^{-1}  \op{curl} \boldsymbol{v} \cdot \op{curl} \overline{\boldsymbol{z}(\ww)}\;\dx-\int_{\omega_i} k^2 \boldsymbol{v} \cdot \overline{\boldsymbol{z}(\ww)} \;\dx \\&- \int_{\partial\omega_i}\epsilon^{-1}(\op{curl} \boldsymbol{v}\times \nn) \cdot \overline{\boldsymbol{z}(\ww)}\;\mathrm{d}s\\
			&=\int_{\omega_i}\boldsymbol{v} \cdot\mathcal{L}_i (\overline{\boldsymbol{z}(\ww)})\;\dx - \int_{\partial\omega_i}\epsilon^{-1}(\op{curl} \boldsymbol{v}\times \nn) \cdot \overline{\boldsymbol{z}(\ww)}\;\mathrm{d}s\\
			&-\int_{\partial\omega_i}\epsilon^{-1} (\boldsymbol{v}\times \nn) \cdot\op{curl} \overline{\boldsymbol{z}(\ww)}\;\mathrm{d}s
			=0.
		\end{aligned}
	\end{equation*}
Using the boundary conditions $\boldsymbol{z}(\ww) \times \nn = \zzz$ and $\mathcal{B}\vv=\boldsymbol{f}$ on $\partial\omega_i$, we derive
\begin{align*}
\int_{\partial\omega_i}\epsilon^{-1}(\op{curl} \boldsymbol{v}\times \nn) \cdot \overline{\boldsymbol{z}(\ww)}\;\mathrm{d}s&=\int_{\partial\omega_i}\epsilon^{-1}\op{curl} \boldsymbol{v} \cdot(\nn\times \overline{\boldsymbol{z}(\ww)})\;\mathrm{d}s=0,\\
\int_{\partial\omega_i}\epsilon^{-1} (\boldsymbol{v}\times \nn) \cdot\op{curl} \overline{\boldsymbol{z}(\ww)}\;\mathrm{d}s&=\int_{\partial\omega_i}\epsilon^{-1} \boldsymbol{v}_T \cdot(\nn\times\op{curl} \overline{\boldsymbol{z}(\ww)})\;\mathrm{d}s\nonumber\\
&=ik^{-1}\int_{\partial\omega_i}\epsilon^{-2}\left(\nn\times\op{curl} \boldsymbol{v}\right)\cdot(\nn\times\op{curl} \overline{\boldsymbol{z}(\ww)})\;\mathrm{d}s\\
&+ik^{-1}\int_{\partial\omega_i}\epsilon^{-1}\boldsymbol{f}\cdot(\nn\times\op{curl} \overline{\boldsymbol{z}(\ww)})\;\mathrm{d}s.
	 \end{align*}
	Combining the equalities above and rearranging the terms, we obtain \eqref{eq:veryweak-weakform}.
\end{proof}


The remaining of this section is devoted to proving the well-posedness of the nonstandard variational formulation \eqref{eq:veryweak-weakform} and deriving the {\it a priori} error estimate. To this end, one has to first derive the $L^2(\partial\omega_i)$-estimate of the normal trace $\nn\times\op{curl}\zz$ for any $\boldsymbol{z}\in X(\omega_i)$. This is established in the following theorem.
\begin{lemma}\label{thm:pw-Regularity}
Let $\boldsymbol{w}\in \mathbf{H}(\mathrm{curl};\partial\omega_i)$ and let $\boldsymbol{z}:=\zz(\ww)\in X(\omega_i)$ satisfy \eqref{eq:pde-dual}. 
Then $\nn\times\op{curl}{\zz}\in L^2(\partial\omega_i)^3$ and 
\begin{align*}
\Vert\nn\times\op{curl}\boldsymbol{z}\Vert_{L^2(\partial\omega_i;\epsilon^{-1})}\lesssim{\eta_i} H^{1/2}k^{-1/2}\Vert\boldsymbol{w}\Vert_{\mathbf{H}(\op{curl};\partial\omega_i)}
\end{align*}
with $\eta_i\coloneqq \Lambda_i+\bar{\epsilon}_{i}\delta_i H$.
\end{lemma}
\begin{proof}
Note that $\boldsymbol{z}\in X(\omega_i)$ satisfies
\begin{equation*}
    \forall \boldsymbol{q}\in X(\omega_i): \int_{\omega_i}\mathcal{L}_i^{*}(z) \cdot\boldsymbol{\bar{q}}\;\dx=\int_{\omega_i}\boldsymbol{w}\cdot\boldsymbol{\bar{q}}\dx.
\end{equation*}
Taking $\boldsymbol{q}:=\boldsymbol{z}$ and applying integration by parts, we arrive at
\begin{align*}
\int_{\omega_i}\epsilon^{-1
}|\op{curl}\boldsymbol{z}|^2\dx&=\int_{\omega_i}k^2|\boldsymbol{z}|^2\dx+\int_{\omega_i}\boldsymbol{w}\cdot \boldsymbol{\bar{z}}\dx.
\end{align*}
Then an application of the Cauchy–Schwarz inequality reveals,
\begin{equation*}  
\bar{\epsilon}_{i}^{-1}\Vert\op{curl}\boldsymbol{z}\Vert_{L^2(\omega_i)}^2\leq\Vert\op{curl}\boldsymbol{z}\Vert_{L^2(\omega_i;\epsilon^{-1/2})}^2\leq k^2 \Vert\boldsymbol{z}\Vert_{L^2(\omega_i)}^2+\Vert\boldsymbol{w}\Vert_{L^2(\omega_i)}\cdot\Vert\boldsymbol{z}\Vert_{L^2(\omega_i)}.
\end{equation*}
Note that $\boldsymbol{z}$ satisfies \eqref{eq:pde-dual}, together with the fact that $\boldsymbol{w}\in \mathbf{H}(\operatorname{div }0;\omega_i)$, implying that $\boldsymbol{z}$ is divergence free, i.e., $\boldsymbol{z}\in \mathbf{H}(\text{div }0,\omega_i)$. Applying the so-called Maxwell estimate for vector fields in \cite[Lemma 3.4]{MR548867} and \cite[Corollary 3.2]{schweizer2018friedrichs}, we derive
\begin{align*} 
 \bar{\epsilon}_{i}^{-1}\Vert\op{curl}\boldsymbol{z}\Vert_{L^2(\omega_i)}^2&\leq\Cpoin{\omega_i}(Hk)^2 \Vert\op{curl}\boldsymbol{z}\Vert_{L^2(\omega_i)}^2\\&+\Cpoin{\omega_i}^{1/2}H\Vert\boldsymbol{w}\Vert_{L^2(\omega_i)}\cdot\Vert\op{curl}\boldsymbol{z}\Vert_{L^2(\omega_i)}.
\end{align*}
This further yields
\begin{align}\label{eq:gradient}
	\Vert\op{curl}\boldsymbol{z}\Vert_{L^2(\omega_i)}\leq\Ce\bar{\epsilon}_{i}\Cpoin{\omega_i}^{1/2}H \Vert\boldsymbol{w}\Vert_{L^2(\omega_i)},
\end{align}
where $\Ce$ is defined in \eqref{ass:res-constant}. 

Equation \eqref{eq:pde-dual} together with \eqref{ass:res-constant}, \eqref{eq:gradient}  and the Maxwell estimate for vector fields, leads to
\begin{align}
\Vert\op{curl}(\epsilon^{-1}\op{curl}\boldsymbol{z})\Vert_{L^2(\omega_i)}&\leq k^2\Vert\boldsymbol{z}\Vert_{L^2(\omega_i)}+\Vert\boldsymbol{w}\Vert_{L^2(\omega_i)}\nonumber\\	
	&\leq k^2\Cpoin{\omega_i}^{1/2}H\Vert\op{curl}\boldsymbol{z}\Vert_{L^2(\omega_i)}+\Vert\boldsymbol{w}\Vert_{L^2(\omega_i)}\nonumber\\
	&\leq (\Ce\Cpoin{\omega_i}\bar{\epsilon}_i(Hk)^2+1)\Vert\boldsymbol{w}\Vert_{L^2(\omega_i)}\nonumber\\
	&< (\Ce+1)\Vert\boldsymbol{w}\Vert_{L^2(\omega_i)}.\label{eq:finalxxxxx}
\end{align}
Since $\boldsymbol{z}\in \mathbf{H}_0(\operatorname{curl};\omega_i)$ and $\epsilon$ is a scalar, by the de Rham diagram \cite[Section 3.7]{Monk2003}, we derive 
\begin{align*}
\epsilon^{-1}\op{curl}\boldsymbol{z}\in \mathbf{H}(\operatorname{curl};\omega_i)\cap\mathbf{H}_0(\operatorname{div };\omega_i).
\end{align*}
Hence, Friedrichs second inequality on convex domain \cite{MR753060} implies
$\epsilon^{-1}\op{curl}\boldsymbol{z}\in H^1(\omega_i)$, and thus $\nn\times\op{curl}\boldsymbol{z}\in H^{1/2}(\partial\omega_i;\epsilon^{-1})$. Friedrichs second inequality on convex domain \cite{MR753060}, combining with \eqref{eq:gradient} and \eqref{eq:finalxxxxx}, implies
\begin{align}
\|\nabla(\epsilon^{-1}\op{curl}\boldsymbol{z})\|_{L^2(\omega_i)}&\lesssim \|\nabla\cdot(\epsilon^{-1}\op{curl}\boldsymbol{z})\|_{L^2(\omega_i)}+\|\op{curl}(\epsilon^{-1}\op{curl}\boldsymbol{z})\|_{L^2(\omega_i)}\nonumber\\
&\lesssim \delta_i \Vert\op{curl}\boldsymbol{z}\Vert_{L^2(\omega_i)}+(\Ce+1)\Vert\boldsymbol{w}\Vert_{L^2(\omega_i)}\nonumber\\
&\lesssim (\bar{\epsilon}_i\delta_i H+1)\Vert\boldsymbol{w}\Vert_{L^2(\omega_i)}.\label{eq:finalyyyy}
\end{align}
Moreover, a combination of the trace inequality \cite[Theorem 1.5.1.10]{grisvard2011elliptic} and  \eqref{eq:gradient} leads to
\begin{align*}
\Vert\nn\times\op{curl}\boldsymbol{z}\Vert_{L^2(\partial\omega_i;\epsilon^{-1})}
&\lesssim H^{-1/2}\|{\epsilon^{-1}\op{curl}\boldsymbol{z}}\|_{L^2(\omega_i)}
+H^{1/2}\|\nabla(\epsilon^{-1}\op{curl}\boldsymbol{z})\|_{L^2(\omega_i)}\\
&\lesssim \bar{\epsilon}_i(\underline{\epsilon}_{i}^{-1}+\delta_i H)H^{1/2}\Vert\boldsymbol{w}\Vert_{L^2(\omega_i)}.
\end{align*}
This proves the desired assertion.
\end{proof}

Thanks to this {\it a priori} estimate presented in Lemma \ref{thm:pw-Regularity}, we are ready to state the well-posedness of the nonstandard variational formulation \eqref{eq:veryweak-weakform}.
\begin{lemma}\label{lem:well-posedness-nonstandard}
Let $\ff\in L^2_{\ttt}(\partial\omega_i)^3$ and let the test space $X(\omega_i)$ be defined in \eqref{eq:test-space}. Then the nonstandard variational form \eqref{eq:veryweak-weakform} is well posed.
\end{lemma}
	\begin{proof}
For any $\boldsymbol{w}\in \mathbf{H}(\mathrm{curl};\partial\omega_i)$, let $\zz(\ww)$ be the solution to Problem \eqref{eq:pde-dual}. 
It follows from Lemma \ref{thm:pw-Regularity} that
\begin{align}\label{eq:b}
|b(\zz(\boldsymbol{w}))|\lesssim {\eta_i} H^{1/2}k^{-3/2} \Vert\boldsymbol{f}\Vert_{L^2(\partial\omega_i)}\|\boldsymbol{w}\|_{\mathbf{H}(\op{curl};\partial\omega_i)}.
\end{align}
A straightforward calculation leads to 
\begin{align*}
|c(\boldsymbol{v},\boldsymbol{w})|&\lesssim \Vert \boldsymbol{v}\Vert_{L^2(\omega_i)}\Vert \boldsymbol{w}\Vert_{L^2(\omega_i)}+
\eta_i H^{1/2}k^{-1}\Vert \nn\times\op{curl}\boldsymbol{v}\Vert_{L^2(\partial\omega_i;\epsilon^{-1})}\Vert \boldsymbol{w}\Vert_{L^2(\omega_i)} \\	
&\lesssim\left(1+\eta_i H^{1/2}k^{-1}\right)k^{-1/2}\Vert \boldsymbol{v}\Vert_{\mathbf{H}(\mathrm{curl};\partial\omega_i)}\Vert \boldsymbol{w}\Vert_{\mathbf{H}(\op{curl};\partial\omega_i)},
		\end{align*}
		which implies the boundedness of the sesquilinear form $c(\cdot,\cdot)$.  

Next, we prove the coercivity. For all $\boldsymbol{w}\in \mathbf{H}(\mathrm{curl};\partial\omega_i)$, since $\ww,\zz(\ww)\in \mathbf{H}(\mathrm{curl};\omega_i)$, there is some $\boldsymbol{q}\in L^2(\omega_i)^3$ such that $\op{curl}(ik\overline{\ww})=\op{curl}\overline{\zz(\ww)}+ \boldsymbol{q}$. Moreover, there is a constant $\theta\in[0,2\pi)$ such that 
		\begin{align*}
			\Re\left(ik^{-1}e^{i\theta}\int_{\partial\omega_i}\epsilon^{-2}(\nn\times\op{curl}\boldsymbol{w})\cdot (\nn\times\boldsymbol{q})\;\mathrm{d}s\right)\geq 0,
		\end{align*}
		and
		\begin{align*}
			\Re\left(e^{i\theta}\Vert\ww\Vert_{L^2(\omega_i)}^2\right)\gtrsim\Vert\ww\Vert_{L^2(\omega_i)}^2.
		\end{align*}
		Thus, we obtain
		\begin{align*}
			&\Re\left(-ik^{-1}e^{i\theta}\int_{\partial\omega_i}\epsilon^{-2}(\nn\times\op{curl}\boldsymbol{w})\cdot (\nn\times\op{curl}\overline{\zz(\ww)})\;\mathrm{d}s\right)\\
			&=\Re\left(-ik^{-1}e^{i\theta}\int_{\partial\omega_i}\epsilon^{-2}(\nn\times\op{curl}\boldsymbol{w})\cdot (\nn\times(\op{curl}(ik\overline{\ww})-\boldsymbol{q}))\;\mathrm{d}s\right)\\
			&= c\Vert\nn\times\op{curl}\boldsymbol{w}\Vert_{L^2(\partial\omega_i;\epsilon^{-1})}^2+\Re\left(ik^{-1}e^{i\theta}\int_{\partial\omega_i}\epsilon^{-2}(\nn\times\op{curl}\boldsymbol{w})\cdot (\nn\times\boldsymbol{q})\;\mathrm{d}s\right)\\
&\geq c\Vert\nn\times\op{curl}\boldsymbol{w}\Vert_{L^2(\partial\omega_i;\epsilon^{-1})}^2
\end{align*}
 with $c\coloneqq\Re(e^{i\theta})>0$. Hence, we get
\begin{align}\label{left hand}
\Re(e^{i\theta}c(\ww,\ww))
\geq c k^{-1}\Vert\ww\Vert_{\mathbf{H}(\mathrm{curl};\partial\omega_i)}^2,
\end{align}
which proves the coercivity. Hence, the Lax-Milgram Theorem \cite[Lemma 25.2]{ern2021finite} ensures the well-posedness of the nonstandard variational problem \eqref{eq:veryweak-weakform}, and this completes the proof.
	\end{proof}
	Finally, we are ready to present the main result. 
\begin{theorem}\label{lem:very-weak}
Given $\boldsymbol{f}\in L^2_{\ttt}(\partial\omega_i)^3$. Let $\boldsymbol{v}$ be the solution to \eqref{eq:very-weak}, then 
\begin{equation*}
\Vert\vv\Vert_{\mathbf{H}(\mathrm{curl};\partial\omega_i)}\lesssim
\eta_i H^{1/2}k^{-1/2} \Vert\boldsymbol{f}\Vert_{L^2(\partial\omega_i)}.
\end{equation*}
\end{theorem}
\begin{proof}
By taking $\boldsymbol{w}:=\boldsymbol{v}$ in \eqref{left hand}, we obtain from \eqref{eq:b} and \eqref{eq:veryweak-weakform}, 
\begin{equation}\label{right hand}
\begin{aligned}
k^{-1}\Vert\vv\Vert^2_{\mathbf{H}(\mathrm{curl};\partial\omega_i)}&\lesssim \Re(e^{i\theta}c(\vv,\vv))=\Re\left(e^{i\theta}b(\boldsymbol{z}(\boldsymbol{v}))\right)\\
&\leq|b(\boldsymbol{z}(\boldsymbol{v}))|\\
&\lesssim \eta_i H^{1/2}k^{-3/2} \Vert\boldsymbol{f}\Vert_{L^2(\partial\omega_i)}\Vert\boldsymbol{v}\Vert_{\mathbf{H}(\op{curl};\partial\omega_i)},
		\end{aligned}
	\end{equation}
	which implies the desired assertion. 
\end{proof}

\begin{corollary}\label{corollary:very-weak}
	Let $\boldsymbol{f}\in L^2_{\ttt}(\partial\omega_i)^3$ and let $\vv$ be the solution to \eqref{eq:very-weak}. Then there hold
	\begin{align*}
\|{\vv_T}\|_{L^2(\partial\omega_i)}&\leq k^{-1}\|{\ff}\|_{L^2(\partial\omega_i)}\\
	\Vert\op{curl} \vv\Vert_{L^2(\omega_i;\epsilon^{-1/2})}	&\lesssim  \left(\eta_i H^{1/2}+k^{-1/2}\right)\|\boldsymbol{f}\|_{L^2(\partial\omega_i)}.
	\end{align*}
\end{corollary}
\begin{proof}
Testing \eqref{eq:very-weak} with $\bar{\vv}$ and applying integration by parts, we arrive at
	\begin{align*}
		&\quad\int_{\omega_i}\mathcal{L}_i\vv\cdot\bar{\vv}\dx\\
		&=\int_{\omega_i}\op{curl}\left({\epsilon^{-1
		}}\op{curl}\boldsymbol{v}\right)\cdot\bar{\vv}\dx-k^2\int_{\omega_i}\boldsymbol{v}\cdot\bar{\vv}\dx\\
		&=\int_{\omega_i}\epsilon^{-1
		}\op{curl}\boldsymbol{v}\cdot(\op{curl}\bar{\vv})\dx-k^2\int_{\omega_i}\boldsymbol{v}\cdot\bar{\vv}\dx-\int_{\partial\omega_i}{\epsilon^{-1
		}}(\op{curl}\boldsymbol{v}\times\nn)\cdot\bar{\vv}\;\mathrm{d}s\\
		&=\int_{\omega_i}\epsilon^{-1
		}|\op{curl}\boldsymbol{v}|^2\dx-k^2\int_{\omega_i}|\boldsymbol{v}|^2\dx-ik\int_{\partial\omega_i}|\vv_T|^2\;\mathrm{d}s-\int_{\partial\omega_i}\ff\cdot\overline{\vv}\;\mathrm{d}s\\
	&=0.
	\end{align*}
Consequently, we obtain 
\begin{align}\label{eq:1/5}
\int_{\omega_i}\epsilon^{-1}|\op{curl}\boldsymbol{v}|^2\dx-k^2\int_{\omega_i}|\boldsymbol{v}|^2\dx-ik\int_{\partial\omega_i}|\vv_T|^2\;\mathrm{d}s=\int_{\partial\omega_i}\ff\cdot\overline{\vv}\;\mathrm{d}s.
\end{align}
Taking the imaginary part and applying Young's inequality, noticing that $\boldsymbol{f}\in L^2_{\ttt}(\gamma_i)^3$, we derive
\begin{align*}
\|{\vv_T}\|_{L^2(\partial\omega_i)}\leq k^{-1}\|{\ff}\|_{L^2(\partial\omega_i)}.
\end{align*}
Taking the real part of \eqref{eq:1/5}, we obtain
	\begin{align*}
		\int_{\omega_i}\epsilon^{-1
		}|\op{curl}\boldsymbol{v}|^2\dx&=k^2\int_{\omega_i}|\boldsymbol{v}|^2\dx
+\Re\left(\int_{\partial\omega_i}\ff\cdot\overline{\vv}\;\mathrm{d}s\right),
	\end{align*}
	 which combining with the Young's inequality and Theorem \ref{lem:very-weak}, reveals the second assertion.
\end{proof}
\subsection{Approximation properties of $\mathbf{V}_{\text{ms},\ell}$}

To analyze the convergence of Algorithm \ref{algorithm:wavelet}, we first introduce the global projection operator $\mathcal{P}_{\ell}$ of level $\ell$: $\mathbf{H}_{\op{imp}}(\op{curl};D) \to \mathbf{V}_{\text{ms},\ell}$. Since the multiscale ansatz space $\mathbf{V}_{\text{ms},\ell}$ is generated by the local multiscale space $\mathcal{L}_i^{-1}(\mathbf{V}_{i,\ell})$ by means of the partition of unity \eqref{eq:global-multiscale}, we only need to define local projection operator $\mathcal{P}_{i,\ell}$ of level $\ell$: $\mathbf{H}_{\op{imp}}(\op{curl};\omega_i)\to \mathcal{L}^{-1}_i(\mathbf{V}_{i,\ell})$, which is defined by
\begin{align}\label{eq:projectionEDGE}
	\mathcal{P}_{i,\ell}(\vv):=
	\sum\limits_{k=1}^{m_i}\sum\limits_{n=1}^{2}\sum\limits_{l,q=1}^{2^\ell}\langle\mathcal{B}\vv,\Psi^{i,k}_{l,q}\ttt_n^{i,k}\rangle_{\partial\omega_i^k}\mathcal{L}^{-1}_{i,n}(\Psi^{i,k}_{l,q}) .
\end{align}
Note that $\mathcal{B}\left(\mathcal{P}_{i,\ell}(\vv)\right)$ is the $L^2(\partial\omega_i)$-projection of $\mathcal{B}\vv$ onto $\mathbf{V}_{i,\ell}$. Note also that any $\vv\in \mathbf{H}_{\op{imp}}(\op{curl};D)$ can be expressed by
\begin{align*}
	\vv=\sum_{i=1}^N \left(\chi_i \vv|_{\omega_i}\right), 
\end{align*}
then the global projection $\mathcal{P}_{\ell}$ of level $\ell$: $ \mathbf{H}_{\op{imp}}(\op{curl};D)\to \mathbf{V}_{\text{ms},\ell}$ can be defined by means of the local projection \eqref{eq:projectionEDGE},
\begin{align}\label{eq:glo-proj}
	\mathcal{P}_{\ell}(\vv)\coloneqq\sum_{i=1}^N\chi_i (\mathcal{P}_{i,\ell}\vv|_{\omega_i}).
\end{align}
\begin{lemma}[Approximation properties of $\mathcal{P}_{\ell}$]\label{projection property}
Let Assumption \ref{ass:regularity} hold. Given $ \ell\in \mathbb{N}_{\geq 0}$ and $s\in (0,1/2]$. Let $\uu\in \mathbf{H}_{\op{imp}}(\mathrm{curl};D)$ be the solution to Problem \eqref{eq:bilinear-form} with $\uu_{I}$ defined in \eqref{eq:harmonic-glo}. Then there holds
\begin{align}
\Vert\uu_I-\mathcal{P}_{\ell}\uu_I\Vert_{L^2(D)}&\lesssim\eta_i H^{s+1/2}k^{-1}2^{-s\ell} 
\|\uu\|_{\mathbf{H}^{s+1/2}(\op{curl};D)}\label{eq:glo-l2}\\
\Vert\op{curl}(\uu_I-\mathcal{P}_{\ell}\uu_I)\Vert_{L^2(D;\epsilon^{-1/2})}
&\lesssim\eta_i H^{s-1/2}k^{-1}2^{-s\ell}\|\uu\|_{\mathbf{H}^{s+1/2}(\op{curl};D)}\label{eq:glo-energy}\\
\Vert\left(\uu_I-\mathcal{P}_{\ell}\uu_I\right)_T\Vert_{L^2(\partial D)}
&\lesssim H^sk^{-1}2^{-s\ell}\|\uu\|_{\mathbf{H}^{s+1/2}(\op{curl};D)}.\label{eq:glo-e_T}
\end{align}
\end{lemma}
\begin{proof}
	Let $\boldsymbol{e}:=\uu_I-\mathcal{P}_{\ell}\uu_I$ be the global error, then the property of the partition of unity of $\{\chi_i\}_{i=1}^{N}$, together with \eqref{eq:harmonic-glo}, leads to
	\[
	\boldsymbol{e}=\sum\limits_{i=1}^{N}\chi_i\boldsymbol{e}^i \qquad\text{ with }
	\qquad\boldsymbol{e}^i:=\uu^{i}_I-\mathcal{P}_{i,\ell}\uu^{i}_I.
	\]
	Our proof is composed of three steps.
	
	\noindent Step 1. Estimate the local error $\boldsymbol{e}^i$ over each local inner boundary $\gamma_i$.  Using the projection error for Haar bases \eqref{prop:approx-wavelets} and a scaling argument, we obtain	
		\begin{align}
			\|\mathcal{B}(\mathcal{P}_{i,\ell} \uu^i_I-I)\|_{L^2(\partial\omega_i)}&=\left\|\mathcal{B}\uu^{i}_I-\sum\limits_{k=1}^{m_i}\sum\limits_{n=1}^{2}\sum\limits_{l,q=1}^{2^\ell}\langle\mathcal{B}\uu_I^i,\Psi^{i,k}_{l,q}\ttt_n^{i,k}\rangle_{\partial\omega_i^k}\psi_{l,m}^{i,k}\ttt_n^{i,k}\right\|_{L^2(\partial\omega_i)}\nonumber\\
			&\lesssim 2^{-s\ell} H^s |\mathcal{B}\uu^{i}_I|_{\mathbf{H}^{s}(\partial\omega_i)}.\label{eq:1}
		\end{align}

	\noindent Step 2. Estimate the local error $\boldsymbol{e}^i$ over each coarse neighborhood $\omega_i$ mainly by the transposition method established in Section \ref{sec:appendix}.
	
	Note that each local error $\boldsymbol{e}^i$ satisfies the following equation,
	\begin{equation}\label{eq:pde-very 1}
		\left\{\begin{aligned}
			&\mathcal{L}_i \boldsymbol{e}^i=\zzz && \text{ in } \omega_i,\\
			&\mathcal{B}\boldsymbol{e}^i=\mathcal{B}(\uu^i_I-\mathcal{P}_{i,\ell} \uu^i_I) &&\text{ on }\partial\omega_i.
\end{aligned}\right.
	\end{equation}
	Theorem \ref{lem:very-weak} and Corollary \ref{corollary:very-weak} indicate that the local error $\boldsymbol{e}^i$ within $\omega_i$ and its tangential trace can be bounded by its boundary data,
	\begin{align*}
		\|\boldsymbol{e}^i\|_{L^2(\omega_i)}&\lesssim
\eta_i H^{1/2}k^{-1} \Vert\mathcal{B}\uu^i_I-\mathcal{B}(P_{i,\ell} \uu^i_I)\Vert_{L^2(\partial\omega_i)},\\
\|\op{curl}\boldsymbol{e}^i \|_{L^2(\omega_i;\epsilon^{-1/2})}&\lesssim\left(\eta_i H^{1/2}+k^{-1/2}\right)\Vert\mathcal{B}\uu^i-\mathcal{B}(P_{i,\ell} \uu^i)\Vert_{L^2(\partial\omega_i)},\\
		\|\boldsymbol{e}^i_T\|_{L^2(\partial\omega_i)}&\leq
		k^{-1} \Vert\mathcal{B}(\uu^i_I-\mathcal{P}_{i,\ell} \uu^i_I)\Vert_{L^2(\partial\omega_i)}.
	\end{align*}
	Then together with \eqref{eq:1}, these lead to
	\begin{align}
		\|\boldsymbol{e}^i\|_{L^2(\omega_i)}&\lesssim\eta_i H^{s+1/2}k^{-1}2^{-s\ell}  \|\mathcal{B}\uu^{i}\|_{\mathbf{H}^{s}(\partial\omega_i)},\label{eq:111111}\\
		\|\op{curl}\boldsymbol{e}^i \|_{L^2(\omega_i;\epsilon^{-1/2})}&\lesssim \left(\eta_i H^{1/2}+k^{-1/2}\right) H^{s}2^{-s\ell} \|\mathcal{B}\uu^{i}\|_{\mathbf{H}^{s}(\partial\omega_i)},\label{eq:333333}\\
		\|\boldsymbol{e}^i_T\|_{L^2(\partial\omega_i)}&\lesssim H^sk^{-1}2^{-s\ell}  \|\mathcal{B}\uu^{i}\|_{\mathbf{H}^{s}(\partial\omega_i)}.\label{eq:111111_2}
	\end{align}
	\noindent Step 3. Estimate the global error by summation of local error. Using the finite overlapping condition \eqref{eq:overlap} and the properties of the partition of unity functions \eqref{eq:pou}, we obtain
	\begin{align*}
	\|\boldsymbol{e}\|_{L^2(D)}&\lesssim\left(\sum_{i=1}^N\|\boldsymbol{e}^i\|_{L^2(\omega_i)}^2\right)^{1/2}.
	\end{align*}
Combining with \eqref{eq:111111}, trace inequality ($\|\mathcal{B}\uu^{i}\|_{\mathbf{H}^{s}(\partial\omega_i)}\lesssim \|\uu\|_{\mathbf{H}^{s+1/2}(\op{curl};\omega_i)}$), and the finite overlapping condition \eqref{eq:overlap} leads to
		\begin{align*}
		\|\boldsymbol{e}\|_{L^2(D)}
&\lesssim \eta_i H^{s+1/2}k^{-1}2^{-s\ell} \left(\sum_{i=1}^N\|\mathcal{B}\uu^{i}\|_{\mathbf{H}^{s}(\partial\omega_i)}^2\right)^{1/2}\\
		&\lesssim\eta_i H^{s+1/2}k^{-1}2^{-s\ell} \|\uu\|_{\mathbf{H}^{s+1/2}(\op{curl};D)}.
	\end{align*}
	This proves \eqref{eq:glo-l2}. 

Using again the finite overlapping condition \eqref{eq:overlap} and the properties of the partition of unity function \eqref{eq:pou}, we obtain
	\begin{align}\label{eq:33333}
		\|\op{curl}\boldsymbol{e}\|_{L^2(D;\epsilon^{-1/2})}^2&\lesssim \sum_{i=1}^NH^{-2}\underline{\epsilon}_i^{-1}\|\boldsymbol{e}^i\|_{L^2(\omega_i)}^2
		+\|\op{curl} \boldsymbol{e}^i\|_{L^2(\omega_i;\epsilon^{-1/2})}^2.
	\end{align}
Analogously, we can prove \eqref{eq:glo-energy} by equations \eqref{eq:111111} \eqref{eq:333333}, \eqref{eq:33333} and the fact that $\underline{\epsilon}_i\geq 1$.

	Equation \eqref{eq:glo-e_T} can be derived in a similar manner using \eqref{eq:111111_2},
	\begin{align*}
		\|\boldsymbol{e}_T\|_{L^2(\partial D)}&\lesssim  \left(\sum_{i=1}^N\|\boldsymbol{e}^i_T\|_{L^2(\Gamma_i)}^2\right)^{1/2}\\
		&\lesssim H^sk^{-1}2^{-s\ell}\|\uu\|_{\mathbf{H}^{s+1/2}(\op{curl};D)},
	\end{align*}
	and then we complete the proof.
	
\end{proof}

\begin{remark}
	Following equation \eqref{eq:harmonic-glo}, we have $\uu_I=\sum_{i=1}^{N}\left(\chi_i \uu^i_I\right)$. Since $
	\nabla\cdot\left(\chi_i \uu^i_I\right)=\nabla\chi_i\cdot\uu^i_I+\chi_i\nabla\cdot\uu^i_I$,  the finite overlapping condition \eqref{eq:overlap} and the properties of the partition of unity functions \eqref{eq:pou} indicate that the approximated solution is not divergence free and the error in $L^2$-norm induced by partition of unity function is bounded by
	\begin{align*}
		\Vert\nabla\cdot(\uu_I-\mathcal{P}_{\ell}\uu_I)\Vert_{L^2(D)}&\lesssim\left(\sum_{i=1}^N\|\nabla\chi_i\cdot\boldsymbol{e}^i\|_{L^2(\omega_i)}^2\right)^{1/2}\\
		&\lesssim H^{-1}\|\boldsymbol{e}\|_{L^2(D)}\lesssim \eta_i H^{s-1/2}k^{-1}2^{-s\ell} \|\uu\|_{\mathbf{H}^{s+1/2}(\op{curl};D)}.
	\end{align*}
	
\end{remark}
The approximation property of the ansatz space
$\mathbf{V}_{\text{ms},\ell}$ can be presented in the following result.
\begin{corollary}[Approximation properties of the multiscale space $\mathbf{V}_{\text{ms},\ell}$] \label{thm:last}
	Let Assumptions \ref{ass:regularity} and \ref{ass:resolution} hold and
	assume that $\boldsymbol{g}\in L^2_{\ttt}(\partial D)^3$. Let $\ell\in
	\mathbb{N}_{+}$, $s\in (0,1/2]$ and $\uu\in \mathbf{H}_{\op{imp}}(\mathrm{curl};D)$ be the solution to Problem
	\eqref{eq:model}. Then there holds
	\begin{equation}\label{eq:waveletErrconv1}
		\begin{aligned}
			\inf\limits_{\vv\in \mathbf{V}_{\rm{ms},\ell}}\|\uu_I-\vv\|_{\mathbf{H}_{\op{imp}}(\mathrm{curl};D)}
			&\lesssim \overline{\eta} H^{s-1/2}k^{-1}2^{-s\ell}\|\uu\|_{\mathbf{H}^{s+1/2}(\op{curl};D)}.
		\end{aligned}
	\end{equation}
Here, $\overline{\eta}\coloneqq \max_{i=1,\cdots,N}\left\{\eta_i\right\}$. 
\end{corollary}

\subsection{Error bound}
We present the error bound for Algorithm \ref{algorithm:wavelet}.
\begin{theorem}\label{wellpose}
Let $0\leq\ell\in \mathbb{N}$, $s\in (1/2,1]$, $\uu\in \mathbf{H}_{\op{imp}}(\mathrm{curl};D)$ the solution to Problem \eqref{eq:bilinear-form} and $\uu_{\text{ms},\ell}$ defined in \eqref{eq:ms-soln}. Let \eqref{ass:dual} hold. Moreover, we assume the following three conditions are satisfied,
\begin{equation}\label{mathbf c}
\begin{aligned}
&\max_{i=1,\cdots,N}\{\eta_i\}\ll k\\
&H= k^{-1-\alpha} \text{ for some }\alpha>0\\
&\ell\geq\max\left\{0,(s-1/2)^{-1}\left(\theta+3/2-(1+\alpha)(s-1/2)\right)\log_2 k\right\}.
\end{aligned}
\end{equation}
Then there holds 
\begin{align}\label{eq:quasi-opt}
	\Vert \uu-\uu_{{\rm ms},\ell}\Vert_{\mathbf{H}_{\op{imp}}(\mathrm{curl};D)}\lesssim
	\inf_{\ww_{{\rm ms},\ell}\in \mathbf{V}_{\text{ms},\ell}}\Vert \uu_I-\ww_{{\rm ms},\ell}\Vert_{\mathbf{H}_{\op{imp}}(\mathrm{curl};D)}.
\end{align}	
\end{theorem}
	\begin{proof}		
Let $\mathbf{e}_{\text{ms},\ell}:= \uu_I-\uu_{{\rm ms},\ell}^I$ be the numerical error for \eqref{eqn:weakform_h-romb}. Then by definition, $\mathcal{B}\mathbf{e}_{\text{ms},\ell}=0$ on $\partial D$. By Helmholtz decomposition, there is $\mathbf{e}^0_{\text{ms},\ell}\in \mathbf{H}_{\op{imp}}(\mathrm{curl};D)\cap \mathbf{H}(\op{div }0;D)$ and a scalar $p\in H^1_0(D)$ such that $\mathbf{e}_{\text{ms},\ell}=\mathbf{e}^0_{\text{ms},\ell}+\nabla p$. Here, $p\in H^1_0(D)$ satisfies
\begin{align*}
(\nabla p, \nabla \xi)_{D}=(\mathbf{e}_{\text{ms},\ell}, \nabla \xi)_{D} \;\forall\xi\in H^1_0(D).
\end{align*}  
This implies  
\begin{align}\label{eq:div-grad-est}
\|{\mathbf{e}_{\text{ms},\ell}}\|_{L^2(D)}^2
=\|{\mathbf{e}^0_{\text{ms},\ell}}\|_{L^2(D)}^2
+\|\nabla p\|_{L^2(D)}^2.
\end{align}
Note that we only concern the case when $\mathbf{e}^0_{\text{ms},\ell}\neq \zzz$. Otherwise, assume that $\mathbf{e}_{\text{ms},\ell}=\nabla p$ in $D$ that is curl-free. Together with the fact that $\mathcal{B}\mathbf{e}_{\text{ms},\ell}=\zzz$ on $\partial D$, we derive $(\mathbf{e}_{\text{ms},\ell})_T=\zzz$ on $\partial D$. Hence, a combination of the Galerkin orthogonality, the fact that $\mathbf{V}_{\text{ms},\ell}$ is $\mathbf{H}(\op{curl};D)$-conforming, and \eqref{eq:bilinear-form} leads to
\begin{align*}
-k^2\|\nabla p\|_{L^2(D)}^2=a(\mathbf{e}_{\text{ms},\ell},\mathbf{e}_{\text{ms},\ell})
&=
a(\uu-\uu_{\text{ms},\ell},\uu-\uu_{\text{ms},\ell})\\
&=a(\uu-\uu_{\text{ms},\ell},\uu)\\
&=\overline{a(\uu,\mathbf{e}_{\text{ms},\ell})}\\
&=\langle (\mathbf{e}_{\text{ms},\ell})_T,\boldsymbol{g}\rangle_{\partial D}\\
&=0.
\end{align*}
 This means $\nabla p=\zzz$ and thus $\mathbf{e}_{\text{ms},\ell}=\zzz$. Hence, this proof is done. For this reason, we assume for some $\beta\in [0,1)$, there holds
\begin{align*}
\|\nabla p\|_{L^2(D)}\leq \beta \|{\mathbf{e}_{\text{ms},\ell}}\|_{L^2(D)}.
\end{align*}
This implies
\begin{align}\label{eq:pressure}
\|\nabla p\|_{L^2(D)}\leq \frac{\beta}{\sqrt{1-\beta^2}} \|{\mathbf{e}^0_{\text{ms},\ell}}\|_{L^2(D)}.
\end{align}
To estimate it, we consider the adjoint problem. We introduce $\zz\in\mathbf{H}_{\op{imp}}(\mathrm{curl};D)$ such that
\begin{equation}\label{eq:adjoint-xxxx}
a(\ww,\zz)=(\ww,\mathbf{e}^0_{\text{ms},\ell})_{D} \quad \forall\ww\in\mathbf{H}_{\op{imp}}(\mathrm{curl};D).
\end{equation}
Note that $\nabla\cdot\mathbf{e}^0_{\text{ms},\ell}=0$ implies 
$\zz\in \mathbf{H}_{\op{imp}}(\mathrm{curl};D)\cap \mathbf{H}(\op{div }0;D)$ is divergence free. By taking $\ww\coloneqq \mathbf{e}_{\text{ms},\ell}$ in \eqref{eq:adjoint-xxxx}, we derive 
\begin{align*}
\|{\mathbf{e}^0_{\text{ms},\ell}}\|_{L^2(D)}^2&=(\mathbf{e}_{\text{ms},\ell},\mathbf{e}^0_{\text{ms},\ell})_{D}\\
&
=a(\mathbf{e}_{\text{ms},\ell},\zz).
\end{align*}
Due to the Galerkin orthogonality and the boundedness of the sesquilinear form \eqref{a property 1},  for any $\ww\in\mathbf{V}_{\text{ms},\ell}$, there holds
\begin{equation*}
\begin{aligned}
\|{\mathbf{e}^0_{\text{ms},\ell}}\|_{L^2(D)}^2&=|a(\mathbf{e}_{\text{ms},\ell},\zz-\ww)|\\
&\leq\Vert\mathbf{e}_{\text{ms},\ell}\Vert_{\mathbf{H}_{\op{imp}}(\mathrm{curl};D)}\cdot\Vert\zz-\ww\Vert_{\mathbf{H}_{\op{imp}}(\mathrm{curl};D)}.
\end{aligned}	
\end{equation*}
This implies, 
\begin{equation}\label{term 2}
\|{\mathbf{e}^0_{\text{ms},\ell}}\|_{L^2(D)}^2
\leq \Vert\mathbf{e}_{\text{ms},\ell}\Vert_{\mathbf{H}_{\op{imp}}(\mathrm{curl};D)}\cdot\inf_{\ww\in\mathbf{V}_{\text{ms},\ell}}\Vert\zz-\ww\Vert_{\mathbf{H}_{\op{imp}}(\mathrm{curl};D)}.
\end{equation}
Next, we derive the approximation properties of $\mathbf{V}_{\text{ms},\ell}$ to $\zz$ defined by \eqref{eq:adjoint-xxxx}. First, we introduce bubble functions for each subdomain $\omega_i$, which reads as seeking $\zz_b^i\in H_0(\op{curl};\omega_i)$, s.t., 
\begin{align*}
a(\ww,\zz_b^i)=(\ww,\mathbf{e}^0_{\text{ms},\ell})_{\omega_i}\qquad \forall\ww\in H_0(\op{curl};\omega_i),
\end{align*}
$\zz_r^i\in H_{\op{imp}}(\op{curl};\omega_i)$, satisfying
\begin{align*}
a(\ww,\zz_r^i)=-\langle(\ww)_T,\epsilon^{-1}\op{curl}\zz_b^i\times\nn\rangle_{\partial \omega_i}\qquad \forall\ww\in H_{\op{imp}}(\op{curl};\omega_i)
\end{align*}
and $\zz_{I}^i\in H_{\op{imp}}(\op{curl};\omega_i)$, defined by
\begin{align*}
a(\ww,\zz_I^i)=\langle(\ww)_T,\mathcal{B}\zz\rangle_{\gamma_i}\qquad\forall\ww\in H_{\op{imp}}(\op{curl};\omega_i).
\end{align*}
Then we define the global functions using the partition of unity functions by $\zz_b\coloneqq\sum_{i=1}^N\chi_i\zz_b^i$, $\zz_r\coloneqq\sum_{i=1}^N\chi_i\zz_r^i$ and $\zz_{I}\coloneqq \sum_{i=1}^N\chi_i\zz_{I}^i$, which all belong to $H_{\op{imp}}(\op{curl};D)$. Moreover, they form a splitting of $\zz$, i.e., $\zz= \zz_b+\zz_r+\zz_I$.


 We first estimate $\zz_b$.  
 The proof to Lemma \ref{thm:pw-Regularity} implies local estimates for local bubble parts, 
\begin{equation}\label{eq:local-bubble-xxxxx}
	\begin{aligned}
\|\zz^i_b\|_{L^2(\omega_i)}&\leq \Ce\Cpoin{\omega_i}\bar{\epsilon}_iH^2\|\mathbf{e}^0_{\text{ms},\ell}\|_{L^2(\omega_i)}\\
\|\op{curl}\zz^i_b\|_{L^2(\omega_i)}&\leq \Ce\Cpoin{\omega_i}^{1/2}\bar{\epsilon}_iH\|\mathbf{e}^0_{\text{ms},\ell}\|_{L^2(\omega_i)}.
	\end{aligned}
\end{equation}
 Collecting these two estimates and using Assumption \ref{ass:resolution}, the finite overlapping condition \eqref{eq:overlap} and the properties of the partition of unity functions \eqref{eq:pou}, we derive 
\begin{align}\label{eq:finalAlpha}
\Vert\zz_b\Vert_{\mathbf{H}_{\op{imp}}(\mathrm{curl};D)}
&\leq\Const{b} \Big(1+\max_{i=1,\cdots,N}\{\Lambda_i^{1/2}\bar{\epsilon}_i^{1/2}\}\Big)H\|\mathbf{e}^0_{\text{ms},\ell}\|_{L^2(D)}\nonumber\\
&\leq\Const{b} \overline{\eta}H\|\mathbf{e}^0_{\text{ms},\ell}\|_{L^2(D)}.
\end{align}
Here, $\Const{b}$ denotes a positive constant independent of $H$ and $k$ that changes value from context to context.

Next, we estimate the approximation property of $\mathbf{V}_{\text{ms},\ell}$ to $\zz_r$. By \eqref{eq:finalyyyy}, the trace inequality and the first condition in \eqref{mathbf c}, we derive
\begin{align}
\|\mathcal{B}\zz^i_{r}\|_{H^{1/2}(\partial \omega_i)}&=\|\epsilon^{-1}\op{curl}\zz^i_b\times\nn\|_{H^{1/2}(\partial \omega_i)}\nonumber\\
&\leq 
\Const{b}(\Lambda_iH+1)\Vert\mathbf{e}^0_{\text{ms},\ell}\Vert_{L^2(\omega_i)}\nonumber\\
&\leq 
\Const{b}\Vert\mathbf{e}^0_{\text{ms},\ell}\Vert_{L^2(\omega_i)}.\label{eq:finalzzzz}
\end{align}
Then we approximate $\zz^i_{r}$ by $\mathcal{L}^{-1}_i(\mathbf{V}_{i,\ell})$.  A combination of
Theorem \ref{lem:very-weak}, Corollary \ref{corollary:very-weak}, \eqref{prop:approx-wavelets} and \eqref{eq:finalzzzz}, indicates 
\begin{align*}
&\quad\|\epsilon^{-1/2}\op{curl}((\mathcal{P}_{i,\ell}-I)\zz^i_{r})\|_{L^2(\omega_i)}+k\|(\mathcal{P}_{i,\ell}-I)\zz^i_{r}\|_{L^2(\omega_i)}\\
&\qquad\;\,\qquad\qquad\qquad\qquad\qquad\qquad+k^{1/2}\|((\mathcal{P}_{i,\ell}-I)\zz^i_{r})_T\|_{L^2(\partial\omega_i)}\\
&\lesssim k^{-1/2}\|\mathcal{B}(\mathcal{P}_{i,\ell}\zz_r^i-\zz_r^i)\|_{L^2(\partial\omega_i)}\\
	&\lesssim k^{-1/2}2^{-\ell/2}H^{1/2}\|\mathcal{B}\zz_r^i\|_{H^{1/2}(\partial\omega_i)}\\
	&\leq\Const{b}k^{-1/2}2^{-\ell/2}H^{1/2}\Vert\mathbf{e}^0_{\text{ms},\ell}\Vert_{L^2(\omega_i)}.
	\end{align*}
Together with \eqref{ass:dual}, the finite overlapping condition \eqref{eq:overlap} and the properties of the partition of unity functions \eqref{eq:pou}, we derive
\begin{align*}
\|\op{curl}(\mathcal{P}_{\ell}-I)\zz_{r}\|_{L^2(D;\epsilon^{-1/2})}&\leq\Const{b} k^{-1/2}2^{-\ell/2}H^{1/2}\|\mathbf{e}^0_{\text{ms},\ell}\|_{L^2(D)}\\
\|(\mathcal{P}_{\ell}-I)\zz_{I}\|_{L^2(D)}
&\leq\Const{b} k^{-3/2}2^{-\ell/2}H^{1/2}\|\mathbf{e}^0_{\text{ms},\ell}\|_{L^2(D)}.
\end{align*}
Collecting these two estimates, we derive 
\begin{align}\label{eq:xxx1234}
\Vert(\mathcal{P}_{\ell}-I)\zz_{r}\Vert_{\mathbf{H}_{\op{imp}}(\mathrm{curl};D)}\leq\Const{b} k^{-1/2}2^{-\ell/2}H^{1/2}\|\mathbf{e}^0_{\text{ms},\ell}\|_{L^2(D)}.
\end{align}
We then estimate the approximation property of $\mathbf{V}_{\text{ms},\ell}$ to $\zz_I$. 
Theorem \ref{lem:very-weak}, Corollary \ref{corollary:very-weak}, \eqref{prop:approx-wavelets} and the trace inequality, indicate, 
\begin{align*}
&\quad\|\epsilon^{-1/2}\op{curl}(\mathcal{P}_{i,\ell}-I)\zz^i_{I}\|_{L^2(\omega_i)}+k\|(\mathcal{P}_{i,\ell}-I)\zz^i_{I}\|_{L^2(\omega_i)}\\
	&\qquad\qquad\qquad\qquad\qquad\qquad\qquad+k^{1/2}\|((\mathcal{P}_{i,\ell}-I)\zz^i_{I})_T\|_{L^2(\partial\omega_i)}\\
	&\lesssim k^{-1/2}\|\mathcal{B}(\mathcal{P}_{i,\ell}\zz-\zz)\|_{L^2(\partial\omega_i)}\\
	&\lesssim k^{-1/2}2^{-(s-1/2)\ell}H^{s-1/2}\|\mathcal{B}\zz\|_{H^{s-1/2}(\partial\omega_i)}
	\\
	&\lesssim k^{-1/2}2^{-(s-1/2)\ell}H^{s-1/2}\|\zz\|_{H^{s}_{\op{curl}}(\omega_i)}.
	\end{align*}
Together with \eqref{ass:dual}, the finite overlapping condition \eqref{eq:overlap} and the properties of the partition of unity functions \eqref{eq:pou}, we derive
\begin{align*}
\|\op{curl}(\mathcal{P}_{\ell}-I)\zz_{I}\|_{L^2(D;\epsilon^{-1/2})}&\leq\Const{b}k^{\theta+1/2}2^{-(s-1/2)\ell}H^{s-1/2}\|\mathbf{e}^0_{\text{ms},\ell}\|_{L^2(D)}\\
\|(\mathcal{P}_{\ell}-I)\zz_{I}\|_{L^2(D)}
&\leq\Const{b} k^{\theta-1/2}2^{-(s-1/2)\ell}H^{s-1/2}\|\mathbf{e}^0_{\text{ms},\ell}\|_{L^2(D)}.
\end{align*}
Collecting these two estimates, we derive 
\begin{align*}
\Vert(\mathcal{P}_{\ell}-I)\zz_{I}\Vert_{\mathbf{H}_{\op{imp}}(\mathrm{curl};D)}\leq\Const{b}k^{\theta+1/2}2^{-(s-1/2)\ell}H^{s-1/2}\|\mathbf{e}^0_{\text{ms},\ell}\|_{L^2(D)}.
\end{align*}
Together with the decomposition $\zz\coloneqq \zz_b+\zz_r+\zz_I$, \eqref{term 2}, \eqref{eq:finalAlpha}, \eqref{eq:xxx1234} and the linearity of $\mathcal{P}_{\ell}$, we derive
\begin{align}
\|{\mathbf{e}^0_{\text{ms},\ell}}\|_{L^2(D)}^2
 &\leq \Const{b}\overline{\eta} H^{1/2}k^{-1/2}\|\mathbf{e}_{\text{ms},\ell}\|_{\mathbf{H}_{\op{imp}}(\mathrm{curl};D)}
 \|\mathbf{e}^0_{\text{ms},\ell}\|_{L^2(D)}\nonumber\\
 &+\Vert\mathbf{e}_{\text{ms},\ell}\Vert_{\mathbf{H}_{\op{imp}}(\mathrm{curl};D)}\Vert(\mathcal{P}_{\ell}-I)\zz_{I}\Vert_{\mathbf{H}_{\op{imp}}(\mathrm{curl};D)}\nonumber\\
 &\lesssim \Big(\overline{\eta}k^{-1-\alpha/2}+k^{\theta+1/2}2^{-(s-1/2)\ell}H^{s-1/2}\Big)\nonumber\\
 &\times\|\mathbf{e}_{\text{ms},\ell}\|_{\mathbf{H}_{\op{imp}}(\mathrm{curl};D)}\|\mathbf{e}^0_{\text{ms},\ell}\|_{L^2(D)}. \nonumber
\end{align}
Hence, we derive under condition \eqref{mathbf c},
\begin{align*}
\|{\mathbf{e}^0_{\text{ms},\ell}}\|_{L^2(D)}\leq \frac{\sqrt{1-\beta^2}}{2k}\|\mathbf{e}_{\text{ms},\ell}\|_{\mathbf{H}_{\op{imp}}(\mathrm{curl};D)}.
\end{align*}
Finally, the G$\mathring{a}$rding's inequality \eqref{a property 2} combining with \eqref{eq:pressure} and \eqref{eq:div-grad-est} implies
\begin{align*}
&\Vert\mathbf{e}_{\text{ms},\ell}\Vert_{\mathbf{H}_{\op{imp}}(\op{curl};D)}^2\\
&\leq |a(\mathbf{e}_{\text{ms},\ell},\mathbf{e}_{\text{ms},\ell})|+2k^2\|{\mathbf{e}_{\text{ms},\ell}}\|_{L^2(D)}^2\\
&\leq\Vert\mathbf{e}_{\text{ms},\ell}\Vert_{\mathbf{H}_{\op{imp}}(\mathrm{curl};D)}\cdot\inf_{\ww\in V_{\op{ms},\ell}}\Vert\uu_I-\ww\Vert_{\mathbf{H}_{\op{imp}}(\mathrm{curl};D)}+\frac{2k^2}{1-\beta^2}\|{\mathbf{e}^0_{\text{ms},\ell}}\|_{L^2(D)}^2\\
&\leq\Vert\mathbf{e}_{\text{ms},\ell}\Vert_{\mathbf{H}_{\op{imp}}(\mathrm{curl};D)}\cdot\inf_{\ww\in V_{\op{ms},\ell}}\Vert\uu_I-\ww\Vert_{\mathbf{H}_{\op{imp}}(\mathrm{curl};D)}+1/2\|{\mathbf{e}_{\text{ms},\ell}}\|_{\mathbf{H}_{\op{imp}}(\mathrm{curl};D)}^2.
\end{align*}
We derive
\begin{align*}
\Vert\mathbf{e}_{\text{ms},\ell}\Vert_{\mathbf{H}_{\op{imp}}(\op{curl};D)}\leq 2
\inf_{\ww\in V_{\op{ms},\ell}}\Vert\uu_I-\ww\Vert_{\mathbf{H}_{\op{imp}}(\mathrm{curl};D)}.
\end{align*}
This proved the desired assertion.
\end{proof}
\section{Numerical tests}\label{sec:num}

In this section, we present a set of three-dimensional experiments to assess the performance of Algorithm~\ref{algorithm:wavelet} and to illustrate the practical implications of the a priori estimate in Theorem~\ref{wellpose}. The computational domain is the unit cube $D:=[0,1]^3$.

For the homogeneous medium $\varepsilon(\xx)\equiv 1$, we use the manufactured field
\[
\uu(\xx)=\bigl(0,e^{ikx_1},0\bigr)^\top,
\qquad
\xx=(x_1,x_2,x_3)^\top\in D.
\]
Throughout all tests, the impedance boundary data is $\boldsymbol{g}(\xx;k)\coloneqq\op{curl}\uu\times \nn-ik\uu_T$ on $\partial D$. 

\paragraph{Discretization, reference solver, and error metrics.}
We use the pair of mesh sizes
$H=\frac1{16},\; h=\frac1{64}$, 
so that the effect of the medium and the wavelet level can be examined independently of mesh refinement.
This choice is consistent with the resolution requirements in the analysis.
More precisely, condition \eqref{mathbf c} requires a coarse scale of order $H\leq k^{-1-\alpha}$ for some $\alpha>0$, 
whereas the fine-grid reference discretization must be chosen sufficiently small to suppress the pre-asymptotic pollution of the lowest-order edge-element approximation. In particular, for the range $k\in\{5,10,15\}$ considered here, the fine mesh $h=1/64$ is sufficiently small for all tested wavenumbers, while the fixed coarse scale $H=1/16$ probes both a comfortably resolved regime ($k=5,10$) and a near-threshold regime ($k=15$). 

The reference solution is computed on $\mathcal T_h$ by the standard lowest-order N\'ed\'elec edge element method on hexahedra presented in Section \ref{subsec:fem}. In the present implementation, the size of the resulting sparse linear system is $811200\times 811200$. In particular, the exact manufactured solution is available for the homogeneous medium with $\varepsilon\equiv 1$ and it is used directly as the reference solution. 

For the multiscale approximation $\uu_{{\rm ms},\ell}$ produced by Algorithm~\ref{algorithm:wavelet} at wavelet level $\ell$, we report the relative errors
\begin{align*}
	e_{\mathbf H_{\op{imp}}(\op{curl};D)}
	\coloneqq \frac{\|\uu_{{\rm ms},\ell}-\uu_{\rm ref}\|_{\mathbf H_{\op{imp}}(\op{curl};D)}}
	{\|\uu_{\rm ref}\|_{\mathbf H_{\op{imp}}(\op{curl};D)}},\;
	e_{L^2(D)}
	\coloneqq\frac{\|\uu_{{\rm ms},\ell}-\uu_{\rm ref}\|_{L^2(D)}}
	{\|\uu_{\rm ref}\|_{L^2(D)}}.
\end{align*}
Here, $\uu_{{\rm ms},\ell}$ is computed on the coarse mesh with local multiscale bases indexed by $\ell$, whereas $\uu_{\rm ref}$ resolves the fine scale on $\mathcal T_h$.

The level-$0$ multiscale space consists of the scaling part of the Haar trace basis on coarse faces and yields $52020$ global multiscale basis functions. Hence the size of the corresponding coarse problem is $52020\times 52020$, which is approximately $15.6$ times smaller in dimension than the fine-grid reference system. The level-$1$ enrichment adds the first Haar-detail traces and generates an additional $156060$ basis functions. Thus the enriched space has a total of $208080$ basis functions, corresponding to a sparse system of size $208080\times 208080$, which is still about $3.9$ times smaller in dimension than the fine-grid reference problem. These dimension counts quantify the algebraic reduction achieved by the multiscale construction before any additional iterative acceleration is used.

\paragraph{Test 1: high-contrast permittivity.}
We first consider three media:
the homogeneous case $\varepsilon(\mathbf{x})\equiv 1$,
and two two-phase inclusion cases in which
\[
\varepsilon(\mathbf{x})=1 \quad \text{in } D\setminus\Yueqi{[13/32,19/32]}^3,
\qquad
\varepsilon(\mathbf{x})=\varepsilon_{\rm in}\quad \text{in } \Yueqi{[13/32,19/32]}^3,
\]
with $\varepsilon_{\rm in}\in\{5,10\}$.
Since the fine mesh resolves the interface, the reported errors reflect only the approximation quality of the multiscale space, rather than a geometric under-resolution effect.

The results are summarized in Table~\ref{ta:test1_new}. 
First, the multiscale method is already accurate at level $\ell=0$, and a single enrichment step to $\ell=1$ reduces both norms to the sub-$2\%$ regime for all three wavenumbers in the homogeneous case. This confirms that, when the medium itself is simple, the dominant task of the coarse space is to capture the oscillatory wave content, and the first Haar enrichment is already sufficient to recover the missing trace information with high fidelity. Second, in the moderate-contrast case $\varepsilon_{\rm in}=5$, the level-$0$ errors increase compared with the homogeneous case, but the enrichment remains uniformly effective. The error reduction from $\ell=0$ to $\ell=1$ is systematic for both norms and for all tested $k$. This behavior is fully consistent with the local-global splitting underlying the method: once the coefficient exhibits a jump across an internal interface, the harmonic extension part becomes harder to represent from coarse traces alone, and the additional wavelet detail at level $\ell=1$ becomes necessary to transmit the interfacial fine-scale information across neighboring coarse cells.
Third, for the stronger contrast $\varepsilon_{\rm in}=10$, the gap between $\ell=0$ and $\ell=1$ becomes even more pronounced. The level-$0$ approximation deteriorates substantially as $k$ increases, while the enriched approximation remains comparatively stable. In particular, at $k=15$, the impedance error drops from $18.471\%$ to $6.940\%$, and the $L^2$-error drops from $16.794\%$ to $6.643\%$. This is precisely the regime where the enrichment is expected to matter most: the coarse space without wavelet detail is not sufficiently expressive to encode both the oscillatory phase and the coefficient-induced local distortion, whereas the enriched traces substantially improve the approximation of the harmonic extension component.


\begin{table}[t]
	\centering
\caption{Simulation results for Test~1.
}\label{ta:test1_new}
		\begin{tabular}{ccc|cc|cc}
			\toprule
			\multicolumn{7}{c}{$\varepsilon \equiv 1$}\\
			\midrule
			\multirow{2}{*}{$k$} & \multirow{2}{*}{$H$} & \multirow{2}{*}{$h$}
			& \multicolumn{2}{c|}{$\ell=0$} & \multicolumn{2}{c}{$\ell=1$}\\
			\cmidrule(lr){4-5}\cmidrule(lr){6-7}
			&&& $e_{\mathbf{H}_{\mathrm{imp}}(\mathrm{curl};D)}$ & $e_{L^2(D)}$
			& $e_{\mathbf{H}_{\mathrm{imp}}(\mathrm{curl};D)}$ & $e_{L^2(D)}$\\
			\midrule
			5  & 1/16 & 1/64 & 5.780\%  & 2.855\% & 0.455\% & 0.404\%\\
			10 & 1/16 & 1/64 & 8.769\%  & 5.295\% & 0.957\% & 0.906\%\\
			15 & 1/16 & 1/64 & 11.079\% & 7.567\% & 1.705\% & 1.619\%\\
			\midrule
			\multicolumn{7}{c}{$\varepsilon=1$ in $D\setminus{[13/32,19/32]}^3$ and $\varepsilon=5$ in ${[13/32,19/32]}^3$}\\
			\midrule
			\multirow{2}{*}{$k$} & \multirow{2}{*}{$H$} & \multirow{2}{*}{$h$}
			& \multicolumn{2}{c|}{$\ell=0$} & \multicolumn{2}{c}{$\ell=1$}\\
			\cmidrule(lr){4-5}\cmidrule(lr){6-7}
			&&& $e_{\mathbf{H}_{\mathrm{imp}}(\mathrm{curl};D)}$ & $e_{L^2(D)}$
			& $e_{\mathbf{H}_{\mathrm{imp}}(\mathrm{curl};D)}$ & $e_{L^2(D)}$\\
			\midrule
			5  & 1/16 & 1/64 & 7.081\%  & 5.070\% & 3.209\% & 2.958\%\\
			10 & 1/16 & 1/64 & 9.512\%  & 7.009\% & 4.413\% & 3.946\%\\
			15 & 1/16 & 1/64 & 12.232\% & 9.440\% & 7.197\% & 5.852\%\\
			\midrule
			\multicolumn{7}{c}{$\varepsilon=1$ in $D\setminus{[13/32,19/32]}^3$ and $\varepsilon=10$ in ${[13/32,19/32]}^3$}\\
			\midrule
			\multirow{2}{*}{$k$} & \multirow{2}{*}{$H$} & \multirow{2}{*}{$h$}
			& \multicolumn{2}{c|}{$\ell=0$} & \multicolumn{2}{c}{$\ell=1$}\\
			\cmidrule(lr){4-5}\cmidrule(lr){6-7}
			&&& $e_{\mathbf{H}_{\mathrm{imp}}(\mathrm{curl};D)}$ & $e_{L^2(D)}$
			& $e_{\mathbf{H}_{\mathrm{imp}}(\mathrm{curl};D)}$ & $e_{L^2(D)}$\\
			\midrule
			5  & 1/16 & 1/64 & 8.876\%  & 5.964\%  & 4.656\% & 2.832\%\\
			10 & 1/16 & 1/64 & 11.501\% & 7.542\%  & 6.917\% & 3.852\%\\
			15 & 1/16 & 1/64 & 18.471\% & 16.794\% & 6.940\% & 6.643\%\\
			\bottomrule
		\end{tabular}
\end{table}

\paragraph{Test 2: highly oscillatory permittivity.}
We next consider
\[
\varepsilon_\delta(\mathbf{x})=\varepsilon(\mathbf{x}/\delta),\qquad
\varepsilon(\mathbf{x})
=\frac18\bigl(3+\sin(2\pi x_2)\bigr)\bigl(3+\sin(2\pi x_3)\bigr)\bigl(3+\sin(2\pi x_1)\bigr),
\]
with $\delta\in\{1,1/5,1/10\}$.
As $\delta$ decreases, the coefficient oscillates on finer and finer spatial scales. This is exactly the regime in which an explicit fine-scale resolution on the coarse mesh is infeasible, and hence it provides a stringent test for the proposed trace-based enrichment mechanism.

The results are reported in Table~\ref{ta:test2_new}. The most robust observation is that enrichment from $\ell=0$ to $\ell=1$ improves the approximation in every tested case. Across all three values of $\delta$ and all three values of $k$, the reduction is substantial in both norms, with the $\ell=1$ errors typically lying in the range of roughly $4\%$--$6\%$. What is more interesting is that, once the first wavelet enrichment is included, the resulting errors remain in a comparatively narrow band across $\delta=1,1/5,1/10$. This indicates that the first Haar-detail correction already captures the dominant oscillatory trace information induced by the coefficient. In other words, although the coefficient oscillates on progressively finer scales as $\delta$ decreases, the multiscale ansatz space at level $\ell=1$ is already rich enough to encode the leading effect of those oscillations on the coarse skeleton. 



\begin{table}[t]
	\centering
\caption{
Simulation results for Test~2.
}
	\label{ta:test2_new}
	\begin{adjustbox}{max width=\textwidth}
		\begin{tabular}{ccc|cc|cc}
			\toprule
			\multicolumn{7}{c}{$\delta=1$}\\
			\midrule
			\multirow{2}{*}{$k$} & \multirow{2}{*}{$H$} & \multirow{2}{*}{$h$}
			& \multicolumn{2}{c|}{$\ell=0$} & \multicolumn{2}{c}{$\ell=1$}\\
			\cmidrule(lr){4-5}\cmidrule(lr){6-7}
			&&& $e_{\mathbf{H}_{\mathrm{imp}}(\mathrm{curl};D)}$ & $e_{L^2(D)}$
			& $e_{\mathbf{H}_{\mathrm{imp}}(\mathrm{curl};D)}$ & $e_{L^2(D)}$\\
			\midrule
			5  & 1/16 & 1/64 & 15.530\% & 14.008\% & 3.845\% & 3.784\%\\
			10 & 1/16 & 1/64 & 20.527\% & 19.597\% & 5.619\% & 5.581\%\\
			15 & 1/16 & 1/64 & 27.418\% & 25.617\% & 6.067\% & 6.032\%\\
			\midrule
			\multicolumn{7}{c}{$\delta=1/5$}\\
			\midrule
			\multirow{2}{*}{$k$} & \multirow{2}{*}{$H$} & \multirow{2}{*}{$h$}
			& \multicolumn{2}{c|}{$\ell=0$} & \multicolumn{2}{c}{$\ell=1$}\\
			\cmidrule(lr){4-5}\cmidrule(lr){6-7}
			&&& $e_{\mathbf{H}_{\mathrm{imp}}(\mathrm{curl};D)}$ & $e_{L^2(D)}$
			& $e_{\mathbf{H}_{\mathrm{imp}}(\mathrm{curl};D)}$ & $e_{L^2(D)}$\\
			\midrule
			5  & 1/16 & 1/64 & 18.755\% & 17.659\% & 4.622\% & 4.035\%\\
			10 & 1/16 & 1/64 & 24.578\% & 22.849\% & 6.193\% & 5.406\%\\
			15 & 1/16 & 1/64 & 25.878\% & 24.755\% & 6.341\% & 6.229\%\\
			\midrule
			\multicolumn{7}{c}{$\delta=1/10$}\\
			\midrule
			\multirow{2}{*}{$k$} & \multirow{2}{*}{$H$} & \multirow{2}{*}{$h$}
			& \multicolumn{2}{c|}{$\ell=0$} & \multicolumn{2}{c}{$\ell=1$}\\
			\cmidrule(lr){4-5}\cmidrule(lr){6-7}
			&&& $e_{\mathbf{H}_{\mathrm{imp}}(\mathrm{curl};D)}$ & $e_{L^2(D)}$
			& $e_{\mathbf{H}_{\mathrm{imp}}(\mathrm{curl};D)}$ & $e_{L^2(D)}$\\
			\midrule
			5  & 1/16 & 1/64 & 19.012\% & 17.655\% & 5.392\% & 5.226\%\\
			10 & 1/16 & 1/64 & 20.317\% & 19.672\% & 6.031\% & 5.706\%\\
			15 & 1/16 & 1/64 & 19.879\% & 18.387\% & 5.714\% & 5.678\%\\
			\bottomrule
		\end{tabular}
	\end{adjustbox}
\end{table}

\section{Conclusion}\label{sec:conclusion}
We presented in this paper a novel numerical homogenization method for time-harmonic Maxwell equations in heterogeneous media with large wavenumber. The convergence of this method with respect to the level parameter $\ell$ is established by means of proposing a novel nonstandard variational formulation. Numerical experiments verify its performance for wavenumbers up to 15. Several avenues for further research remain. One key challenge is to obtain accurate solutions for much larger wavenumbers. The primary difficulty stems from the linear system size growing at least cubically with respect to the wavenumber, posing significant challenges in storage and computational solution. A potential approach is to develop a fast solver based on a two-level domain decomposition method, utilizing the designed multiscale ansatz space as the coarse space. Another important question concerns preserving the divergence-free property within the multiscale ansatz space. Currently, due to the multiplication of partition of unity functions in our approach, this property is lost, which may be critical in practical applications. Therefore, we plan to explore the development of a multiscale ansatz space that maintains the divergence-free property without relying on partition of unity functions.
\section*{Acknowledgements}
Part of this research was performed while the authors were visiting the Institute for Mathematical and Statistical Innovation (IMSI), which is supported by the National Science Foundation (Grant No. DMS-1929348).
\bibliographystyle{siamplain}
\bibliography{reference}
\end{document}